\documentclass[a4paper,leqno,12pt]{article}
\usepackage[utf8]{inputenc}
\usepackage[T1]{fontenc}
\usepackage[english]{babel}

\usepackage{crimson}

\usepackage{geometry}

\usepackage{fancyhdr}

\usepackage{hyperref}
\hypersetup{
	colorlinks   = true, %Colours links instead of ugly boxes
	urlcolor     = blue, %Colour for external hyperlinks
	linkcolor    = blue, %Colour of internal links
	citecolor   = black %Colour of citations
}

\usepackage[inline]{enumitem}

\usepackage{graphicx}

\usepackage{color}

\usepackage[intlimits]{amsmath}
\usepackage{amssymb, amsfonts}
\usepackage[thmmarks, amsmath, amsthm]{ntheorem}

\usepackage{MnSymbol}
\usepackage{mathbbol}

\usepackage{bbm}

\usepackage{mathrsfs}

\usepackage[all]{xy}

\usepackage{tikz-cd}

\numberwithin{equation}{section}
\newtheorem{theoremcounter}{theoremcounter}[section]

\theoremstyle{plain}

\newtheorem{corollary}[theoremcounter]{Corollary}
\newtheorem{lemma}[theoremcounter]{Lemma}
\newtheorem{proposition}[theoremcounter]{Proposition}

\newtheorem{theorem}[theoremcounter]{Theorem}

\newtheorem{example}[theoremcounter]{Example}
\newtheorem{question}[theoremcounter]{Question}

\theoremnumbering{Alph}
\newtheorem{introtheorem}{Theorem}
\newtheorem{introcorollary}[introtheorem]{Corollary}
\newtheorem{introdefinition}[introtheorem]{Definition}
\theoremstyle{definition}

\newtheorem{definition}[theoremcounter]{Definition}

\theoremstyle{remark}

\newtheorem{remark}[theoremcounter]{Remark}
\renewcommand{\theenumi}{(\roman{enumi})}
\renewcommand{\labelenumi}{\theenumi}

\usepackage{xargs}                      % Use more than one optional parameter in a new commands
\usepackage[colorinlistoftodos,prependcaption,textsize=tiny]{todonotes}
\newcommandx{\unsure}[2][1=]{\todo[linecolor=red,backgroundcolor=red!25,bordercolor=red,#1]{#2}}
\newcommandx{\change}[2][1=]{\todo[linecolor=blue,backgroundcolor=blue!25,bordercolor=blue,#1]{#2}}
\newcommandx{\info}[2][1=]{\todo[linecolor=OliveGreen,backgroundcolor=OliveGreen!25,bordercolor=OliveGreen,#1]{#2}}
\newcommandx{\improvement}[2][1=]{\todo[linecolor=Plum,backgroundcolor=Plum!25,bordercolor=Plum,#1]{#2}}

\usepackage[nameinlink, capitalize, noabbrev]{cleveref}

\newcommand{\Cs}{\mathrm{C}^*}
\newcommand{\Csr}{\mathrm{C}^*_{\text{red}}}
\newcommand{\Csm}{\mathrm{C}^*_{\text{max}}}
\newcommand{\N}{\mathcal N}
\newcommand{\ra}{\rightarrow}
\newcommand{\bF} {\mathbb F}
\newcommand{\bb} {\mathbb b}
\newcommand{\bC} {\mathbb C}
\newcommand{\bN} {\mathbb N}
\newcommand{\bR} {\mathbb R}
\newcommand{\bQ} {\mathbb Q}
\newcommand{\bZ} {\mathbb Z}
\newcommand{\cM} {\mathcal M}
\newcommand{\cB}{\mathcal B}
\newcommand{\cS}{\mathcal S}
\newcommand{\Tr}{\mathrm {Tr}}
\newcommand{\PSL}{\mathrm{PSL}}
\newcommand{\SL}{\mathrm{SL}}
\newcommand{\GL}{\mathrm{GL}}

\newcommand{\authors}{Sanaz Pooya and Hang Wang}
\renewcommand{\title}{Higher Kazhdan projections and delocalised $\ell^ 2$-Betti numbers}
\begin{document}
	
	%%%%%%%%
	%% Front page
	%%%%%%%%

	\thispagestyle{empty}
	
	\noindent
	\begin{minipage}{\linewidth}
		\begin{center}
			\textbf{\Large \title} \\
			\authors    
		\end{center}
	\end{minipage}
	
	\renewcommand{\thefootnote}{}
	%\footnotetext{last modified on \today}
	%\footnotetext{%
		%\textit{Keywords: higher Kazhdan projections, K-theory, 
  %$\ell^2$-Betti numbers, 
  %delocalised $\ell^2$-Betti numbers}
	%}
	
	\vspace{2em}
	\noindent
	\begin{minipage}{\linewidth}
		\textbf{Abstract}.  We provide an explicit description of the K-classes of higher Kazhdan projections in degrees greater than 0 for specific free product groups and Cartesian product groups.   
		Employing this description, we obtain new calculations of Lott's delocalised $\ell^2$-Betti numbers for groups. Notably, we establish the first non-vanishing results for infinite groups.	 
		
	\end{minipage}

	%%%%%%%%% 
	%% Document body
	%%%%%%%%%
	
	\section{Introduction}
	\label{sec:introduction}
	%In operator algebras K-theory plays an important role as it contains a lot of information about a $\Cs$-algebra. 
   %While it contains plenty of information, it is hard to compute at the same time. 
	Within operator algebras K-theory holds importance as it encapsulates extensive information regarding a $\Cs$-algebra. In view of its richness in information, computational challenges persist.
	In light of these challenges, understanding specific K-classes can provide us with partial information that may lead to improved understanding of an object it is associated with. 
	In the framework of group $\Cs$-algebras there are numerous ways to produce K-classes that can be analysed. 
	The averaging projection in the $\Cs$-algebra of a finite group is one of the most elementary such examples. More involved K-classes can be produced 
	as Bott projections and their various generalisations. Also K-classes arising from generalised Fredholm operators by applying the Baum-Connes assembly map, to certain extend, be considered as more concrete than general K-classes of group $\Cs$-algebras.
	A rather different kind of K-classes was recently introduced under the name of higher Kazhdan projections in \cite{linowakpooya2020}. They can provide a new way to obtain concrete K-classes. 
	Their construction was motivated by that of Kazhdan projections associated to property (T) groups. However, in the generality in which they were defined, they can also be defined for non-property (T) groups.	
	Let us briefly recall the construction and discuss to what extent these classes could be understood so far.
		
	Let $G$ be a discrete group of type $F_{n+1}$. For a suitable unitary representation $(\pi, \mathcal H)$ of $G$, the first author, Li and Nowak defined a sequence of higher Kazhdan projections $(p_n)_{n\ge 0}$. They lie in matrices over $\mathcal B(\mathcal H)$. 
	By construction, $p_n$ is the projection onto harmonic $n$-cochains in a suitable chain complex computing cohomology of $G$ with coefficients in $(\pi, \mathcal H)$, which relates it directly to reduced group cohomology. 
	In order to view $p_n$ inside matrices over $\Cs_{\pi}(G)$ we need a spectral gap for the $n$-th Laplacian $\Delta_n$. This is equivalent to that the $n$-th and $n+1$-st cohomology with coefficients in $\mathcal H$ are reduced  \cite{badernowak2015}. 	
	
	If $\pi$ is the universal representation of a property (T) group $G$, then  $p_0 \in \Cs_{\mathrm {max}} (G)$ is the classical Kazhdan projection.			
	%\todo{not sure: our def would not necessarily give that the projection lives in the maximal-C*- algebra but some matrices over it!}
	The case where $\pi$ is chosen to be the left regular representation 
	$\lambda $ is particularly interesting, since in this case the trace of the K-class of $p_n$ is the $n$-th $\ell^2$-Betti number of the group. 
	%Throughout this article we will be interested in the case of the left regular representation as we will investigate connection between K-classes of higher Kazhdan projections and a generalisation of $\ell^2$-Betti number.
	The examples of non-abelian free groups $\mathbb F_k$ and 
	$\mathrm{PSL}(2, \mathbb Z)$ showcase different behaviors of $p_1$.
	In the example of $\mathbb F_k$, we have that $[p_1] = (k - 1)[1]$ in K-theory, while for $\mathrm{PSL}(2, \mathbb Z)$ we have that $[p_1] \notin \mathbb Z [1]$  due to the non-integral first $\ell^2$-Betti number of the group. 	In contrast to the classical Kazhdan projection which only exists when $G$ has property (T) we now saw that higher Kazhdan projections exist and exhibit different behavior in this more general setting. 
	A natural question to ask is whether it is possible to explicitly describe their K-classes. In this article we answer this question positively by explicitly describing the K-class of higher Kazhdan projections in two classes of groups: free products of finite cyclic groups including $\mathrm {PSL(2, \mathbb Z)}$ and certain Cartesian products. 
	
	By directly computing the space of harmonic  $n$-cochains we are able to obtain very concrete representatives for the K-class of higher Kazhdan projections associated with these groups.
	
	\begin{introtheorem} \label{thm A}
		(Theorem \ref{thm: Z_n*Z_m}) 
		Let $G= \mathbb Z_m \ast \mathbb Z_n$ for $m \geq 2$ and $n \geq 3$. The K-class of first higher Kazhdan projection $p_1$ in $\mathrm K_0(\Csr(G))$ is described by
		\[
			[p_1] = [1] - \left[\frac{1}{m}
			\sum_{0 \leq i < m}{s^i}\right] - 
			\left[\frac{1}{n}{\sum_{0 \leq j <n}{t^j}} \right],
		\]
		where $s$ and $t$ are generators (of order $m$ and $n$) of the cyclic groups, respectively.
	\end{introtheorem}
	
	% Another class of groups where we concretely describe the K-theory class of $p_n$ is that of product of a free group with a finite group. 
	Employing a product construction, we can produce higher Kazhdan projections with non-integral K-class in higher degrees. This is the content of the next theorem.
	
	\begin{introtheorem} \label{thm B}
		(Theorem \ref{thm:higher Kazh proj product group})
		Let $F$ be a finite group let $n \in \mathbb N$ and consider the product
		$G = \mathbb F_2 \times \dotsc \times \mathbb F_2 \times F$ of $n$ factors of $\mathbb F_2$ with $F$. Then the K-class of the $n$-th higher Kazhdan projection $p_n$ of $G$ satisfies		
		\[
		[p_n] = \left[\frac{1}{|F|}{\sum_{g\in F}{g}}\right].
		\]
		%**Another formulation of the theorem
		%Let $\mathbb F_2$ be the free group on two generators and $F$ be a finite group. Let $G = \mathbb F_2 \times \dotsc \times \mathbb F_2 \times F $ be the product of $n$-copies of $\mathbb F_2$ with $F$. K-classes in 
		%$\mathrm K_0 (\Csr(G))$ of higher Kazhdan projections $p_n$ are described by 
		%\[
			%[p_n] = \left[\frac{1}{|F|}{\sum_{g\in F}{g}}\right].
		%\]		
	\end{introtheorem}
	
	Besides providing a clearer picture of this new class of cohomologically defined projections, the relevance of explicitly describing their K-classes stems from their connections to Lott's delocalised $\ell^2$-Betti numbers for manifolds \cite{lott1999}. 
 %Before introducing our cohomological definition, which applies to groups, let us introduce some notation.
	\\
    For $g\in G$ we denote its conjugacy class $\langle{g}\rangle$.	
	Associated with the conjugacy class of $g \in G$ is a delocalised trace, which is the tracial functional, defined as
	\[
		\tau_{\langle g\rangle} \colon \ell^1(G) \ra \mathbb C
	, \qquad  \tau_{\langle g\rangle}(f) = \sum_{h\in {\langle g\rangle}} {f(h)}.
	\]
	  \indent Aiming to pair such delocalised traces with K-classes an obstruction can be observed for infinite conjugacy classes when $\tau_{\langle g\rangle}$
	does not necessarily extend to the $\Cs$-algebra closure in question. This problem has different remedies.	
	To overcome this problem, Lott introduced his delocalised $\ell^2$-Betti numbers as the limit of delocalised traces of heat kernel operators for closed manifolds, whose analogue for groups lies in matrices over $\ell^1(G)$.	 
	More concretely, for a group $G$ denoting by $\Delta_n$ the $n$-th combinatorial  Laplacian, the $n$-th delocalised $\ell^2$-Betti number would be described by the following expression:
	\begin{gather*}
		\beta^{(2)}_{n, \langle g\rangle}(G) = \lim_{t \to \infty} \tau_{\langle g\rangle}(e^{-t\Delta_n}).
	\end{gather*}
	However, in view of the context of higher Kazhdan projections it makes sense to adopt a more K-theoretic definition. 
	%which we will show to agree with Lott's definition whenever it makes sense (see Proposition ?).
	%The inexplicit definition of delocalised $\ell^2$-Betti numbers makes them very hard to compute. 	
	Since $\lim_{t\to \infty}e^{-t\Delta_n}$ considered as a strong limit does not always belong to $\ell^1(G)$, one would need to find a different subalgebra of $\Csr(G)$ to which these traces can extend. 	
	When $G$ is a hyperbolic group, Puschnigg has constructed suitable smooth dense subalgebras of $\Csr(G)$, which have the same K-theory as $\Csr(G)$ and to which delocalised traces extend (see \cite{puschnigg2010}). 
		
	%Pairing higher Kazhdan projections in the case of hyperbolic groups with the delocalised traces
	%, which can be shown they live in some matrices over these smooth subalgebras, 
	%gives rise to the following  definition for delocalised $\ell^2$-Betti numbers.
	For the purpose of this article we then adopt the following definition.
	\begin{introdefinition}
		\label{def C} 
		(Definition \ref{def:deloc l2betti num formula})
		Let $G$ be a group of type $F_{n+1}$, and $g\in G$. Assume that there is a smooth subalgebra $\mathcal S \subset \Csr (G)$ to which the delocalised trace $\tau_{\langle g\rangle}$ extends. Assume further that the K-class of $p_n$ lies in $\mathrm K_0(\Csr (G))$. We define the $n$-th delocalised $\ell^2$-Betti number of $G$ as
		\begin{gather*} 
			\beta ^{(2)}_{n, \langle g\rangle}(G) = \tau_{\langle g\rangle} ([p_n]).
		\end{gather*}
	\end{introdefinition}
	
	As a byproduct of the explicit description of higher Kazhdan projections, we establish the first calculations for non-zero delocalised $\ell^2$-Betti numbers of infinite groups. 
	
	\begin{introcorollary} 
		\label{cor D}
		(Corollary \ref{cor: deloc l2 Betti Zn*Zm})
		The delocalised $\ell^2$-Betti numbers for the group $G = \mathbb Z_m \ast \mathbb Z_n $ are
		\begin{equation*}
			\beta^{(2)}_{1, \langle{g}\rangle}(G) = 
			\begin{cases}
				1-\frac{1}{m}-\frac{1}{n} &\qquad g=e\\
				-\frac{|\langle{g}\rangle|}{m}  &\qquad g \in \mathbb Z_m \setminus \{e\}\\
				-\frac{|\langle{g}\rangle|}{n}  &\qquad g \in \mathbb Z_n \setminus \{e\}\\
				0 \qquad & \qquad  \text{otherwise} 
			\end{cases}
		\end{equation*}
		and $\beta ^{(2)}_{k, \langle{g}\rangle}(G) = 0$ for $k \neq 1$ and all $g\in G$.
	\end{introcorollary}
	
	% We also prove the following non-vanishing result for delocalised $\ell^2$-Betti numbers in higher degree.
	
	% \begin{introcorollary} \label{cor E}
		% (Corollary \ref{cor: deloc l2 Betti product group})
		%     Let $G = \mathbb F_2 \times \dotsc \times \mathbb F_2 \times F $ be the $n$-fold product of the free group $\mathbb F_2$  with a finite group $F$. The $n$-th delocalised $\ell^ 2$-Betti numbers are
		%     \begin{equation*}
			%         \beta^{(2)}_{n, \langle{g}\rangle}(G) = 
			%         \begin{cases}
				% 	        \frac{|\langle{g}\rangle|} {|F|}
				% 	        & \qquad g \in F \\
				% 	        0  &\qquad  \text{otherwise}
				%         \end{cases}
			%     \end{equation*}
		%     They vanish in all other degree.
		% \end{introcorollary}
	
	In addition to establishing non-vanishing results, we also demonstrate several vanishing instances. Notably, we prove that if the $n$-th $\ell^2$-Betti number of a hyperbolic group $G$ vanishes, then the $n$-th delocalised $\ell^2$-Betti numbers $\beta^2_{n, \langle g \rangle} (G)$ vanishes for all $g\in G$ (see Proposition \ref{prop:l2 vanishing}).
	Moreover, we provide more evidences concerning an observation made by Lott on vanishing of delocalised $\ell^ 2$-Betti numbers in the case of torsion-free hyperbolic groups. Specifically, we establish that for such a group $G$ and under the presence of a spectral gap for $\Delta_n$, all $n$-th delocalised $\ell^2$-Betti numbers $\beta^2_{n, \langle g \rangle}(G)$ vanish for $g\neq e$ (see Proposition \ref{prop:torfreehyper}).
	
	\textbf{Added Notes:} After submission of this work, the authors together with B. Ren posted the preprint \cite{pooyarenwang2025}   obtaining related results. The latter establishes a more general framework to explicitly compute the K-theory class of higher Kazhdan projections for the class of virtually free groups and computes the corresponding delocalised $\ell^2$-Betti numbers. This class includes all free products of the kind we discussed in this article. The two methods are however very different. Here we directly analyse $\ker \Delta_n$  to describe the associated higher Kazhdan projection, while in \cite{pooyarenwang2025} we appeal to the combinatorial Euler characteristic of Emerson and Meyer \cite{EM} as well as the structure theory of virtually free groups. The K-theory of higher Kazhdan projections can be then expressed as an alternating sum of averaging projections associated to specific isotropy subgroups of the action on the associated Bass-Serre tree of such a group.
	
	This article contains four sections. After the introduction, in Section \ref{sec:preliminaries} we collect preliminaries on higher Kazhdan projections, delocalised $\ell^2$-Betti numbers and smooth subalgebras. 
	%In particular, we compare  Definition \ref{def C} with the one of Lott. 
	In Section \ref{sec: vanishing}, we prove our vanishing results. In Section \ref{sec: computations} we prove Theorem \ref{thm A} and compute its first delocalised $\ell^2$-Betti numbers as it is stated in Corollary \ref{cor D}. Moreover, we prove Theorem \ref{thm B}.
	
	\subsection*{Acknowledgments}
	H.W. is supported by the grants NSFC 12271165,23JC1401900 and in part by Science and Technology Commission of Shanghai Municipality (No. 22DZ2229014). S.P. was supported by Knut and Alice Wallenberg foundation through grant number 31001288. We would like to thank John Lott and Aaron Tikuisis for very useful comments and Rufus Willett for various conversations around the topic.
	
	\section{Preliminaries}
	\label{sec:preliminaries}
	\subsection{Higher Kazhdan projections} \label{subsec:higher Kazh proj}
	In this section we briefly recall the construction of higher Kazhdan projections from \cite{linowakpooya2020}. 
	% We will be working in the setting of simplicial complexes however all of our considerations can be carried out equally in the setting of CW-complexes. 
	Let $G$ be a group of type $F_{n+1}$ with a chosen model $X$ of $\mathrm B G$ with finite $n+1$-skeleton.
	Let $(\pi, \mathcal H)$ be a unitary representation of $G$. We consider the group cohomology of $G$ with coefficients in $\mathcal H$ computed via 
	\[
	\mathrm{C}^0(G, \mathcal H) 
	\stackrel{d_0} \longrightarrow \mathrm{C}^1(G, \mathcal H)
	\stackrel{d_1}  \longrightarrow \cdots 
	\stackrel{d_{i-2}}\longrightarrow
	\mathrm{C}^{i-1}(G, \mathcal H) 
	\stackrel{d_{i-1}} \longrightarrow  \mathrm{C}^i(G, \mathcal H)
	\stackrel{d_i}  \longrightarrow  \mathrm{C}^{i+1}(G, \mathcal H)
	\stackrel{d_{i+1}} \longrightarrow  \cdots.
	\]
	By the assumption that $G$ is of type $F_{n+1}$ we may consider the following identification when $i \leq n$
	\[
	\mathrm{C}^i(G, \mathcal H) = \left\{ f \colon X^{(i)} \ra \mathcal H\right\} \cong \mathcal H ^{\oplus k_i},
	\]
	with $k_i$ being the number of cells in $X^{(i)}$.
	The coboundaries $d_i$'s are identified with matrices in 
	$\mathrm {M}_{k_{i+1} \times k_i}(\mathbb{Z} G)$ over the integral group ring $\mathbb {Z}G$.
	
	The (combinatorial) Laplacian in degree $n$ is the operator
	\[
	\Delta_i = d_i^* d_i + d_{i-1} d^*_{i-1} \in \mathrm M_{k_{i}}(\mathbb {Z}G).
	\]
	
	The kernel of $\pi(\Delta_i)$ is the space of harmonic $i$-cochains for $\pi$, which is by the Hodge-de Rham isomorophism identified with the reduced cohomology ${\overline{\mathrm H}}^n(G, \mathcal H)$, that is
	\[
	\ker \pi(\Delta_i) \simeq {\overline{\mathrm H}}^i(G, \mathcal H). 
	\]
	Recall that in the setting where the cochain spaces in a cochain complex are equipped with a Banach space structure the corresponding reduced cohomology group are defined by the quotients 
	${\overline{\mathrm H}}^i(G, \mathcal H) :=\ker d_i ~/~ \overline{\mathrm{im} \, d_{i-1}}$, where $\overline{\mathrm{im}\,d_{i-1}}$ is the closure of the image of the codifferential $d_{i-1}$. The cohomology is reduced in degree $i$ if 
	${\overline{\mathrm H}}^i(G, \mathcal H) = \mathrm {H}^i(G, \mathcal H)$.
	
	\begin{definition} \cite{linowakpooya2020}
		Let $G$ be a group of type $F_{n+1}$ with a chosen model $X$ for $\mathrm B G$. Let $(\pi, \mathcal H)$ be a unitary representation of $G$. The projection $p_n \colon \mathcal H ^{\oplus n_k} \ra \ker \pi(\Delta_{n})$ is called a higher Kazhdan projection in degree $n$, where $n_k$ is the number of cells in the $n$-th skeleton of $X$.	
	\end{definition}
	
	Similar to the situation with property (T), higher Kazhdan projections exist in the presence of a spectral gap at $0$.
	% Higher Kazhdan projections lives a priory in an amplification of the von Neumann algebra generated by the given representation. Assuming spectral gap, one may view them as projections belonging to an amplification of the $\Cs$-algebra generated by the given representation. We will verify the spectral gap assumption by employing the following result.
	
	\begin{lemma} 
		{\cite [Proposition 16.2] {badernowak2020}} 
		\label{lem:spectral gap condition}
		Let $(\pi, \mathcal H)$ be a unitary representation of $G$. The Laplacian $\pi(\Delta_n)$ has spectral gap at $0$ if and only if the cohomology is reduced in degree $n$ and $n+1$.
	\end{lemma}
	
	Assume that for the left regular representation $\lambda$, the operator $\lambda(\Delta_n)$ has spectral gap at $0$ so that $[p_n] \in \mathrm K_0 (\Csr(G))$. Let $\tau$ denote the canonical trace on $\Csr(G)$ and by abuse of notation its extension to the K-theory level. By \cite[Proposition 7]{linowakpooya2020}, the $n$-th $\ell^2$-Betti numbers of $G$ are obtained by the trace pairing  
	\begin{gather} \label{def:l2 betti number definition}
		\beta^{(2)}_n (G) = \tau ([p_n]).
	\end{gather}
	
	\textbf{Convention}. We will exclusively work with the left regular representation. Further, By abuse of notation, we identify $\lambda(\Delta_n)$ as $\Delta_n$.
	
	\subsection{Delocalised 
		\texorpdfstring{$\ell^2$}{l²}
		-Betti numbers}\label{subsec:deloc l2}
	John Lott introduced  delocalised $\ell^2$-Betti numbers as part of his delocalised $\ell^2$-invariants \cite{lott1999}. These are analogues of the 
	$\ell^2$-invariants of closed Riemannian manifolds. Before presenting the definition, we recall the notion of delocalised traces that will be one of its ingredients.
	
	\begin{definition}
		Let $G$ be a discrete group and $g\in G$. The delocalised trace  $\tau_{\langle g\rangle}$ on $\ell^1(G)$ is the bounded linear map with tracial property 
		\[
		\tau_{\langle g\rangle}: \ell^1(G)\rightarrow \mathbb C, \qquad 
		\tau_{\langle g\rangle}(f)= \sum_{h\in \langle g\rangle}f(h),
		\]
		where $\langle g\rangle$ denotes the conjugacy class of $g$.
	\end{definition}
    Precomposition with matrix traces makes it possible to consider delocalised traces as functionals on matrix algebras.
	
	The following analog is in complete analogy with Lott's definition for closed  manifolds. The difference is that we consider groups rather than manifolds and combinatorial Laplacian instead of Laplace-Baltrami operators. 
	
	% When the classifying of a group has a homotopy type of a manifold then the two formulas coincide.
	
	\begin{definition}  \label{def: lott deloc betti numbers} {(cf.~ \cite[Definition 3]{lott1999})}
		Let $G$ be a discrete group. For $g\in G$, the $n$-th delocalised $\ell^2$-Betti number of $G$ with respect to the conjugacy class $\langle g\rangle$ is defined by
		\begin{gather}\label{analytic def}
		\beta ^{(2)}_{n, \langle g\rangle}(G) = \lim_{t \to \infty} \tau_{\langle g\rangle}(e^{-t \Delta_n}),
		\end{gather}
		when the limit exists.
	\end{definition}
	We refer to this definition as the analytic definition of delocalised $\ell^2$-Betti numbers.
 
	%  the new below remark by Hang added in July which is the same as what we have had in lines above Def 2.4. So I delete it.
	% \begin{remark}
		% More precisely, replace the closed manifold $M$ by the classifying space $BG$ as a CW complex and replace the group associated to a normal covering $\tilde M$ by the fundamental group $G$ of $BG$. If moreover $BG$ is a closed manifold, then we have \[
		% \beta^{(2)}_{n, \langle g\rangle}(G)=b^{(2)}_{n, \langle g\rangle}(BG).
		% \]
		% where the right hand side denotes Lott's delocalised $\ell^2$-Betti number. 
		% In general, $BG$ is not a manifold, and it would be more desirable to use the combinatorial Laplacians. To see the two notions coincide, we will treat it together with the claim in the above display in a forthcoming paper. 
		% \end{remark}
	
	In analogy with Lott's observation for manifolds in \cite[Proposition 18]{lott1999}, self-adjointness of the operator $e^{-t\Delta_n}$ together with $\tau_{\langle g\rangle}(a^*) = \overline{{\tau_{\langle g^{-1}\rangle}{(a)}}}$ imply that 
	\[
	\beta ^{(2)}_{n, \langle g^{-1}\rangle}(G)=\overline{\beta ^{(2)}_{n, \langle g\rangle}(G)}.
	\]
	% Compare this with Proposition 18 of \cite{lott1999}.
	
	% \begin{remark}
		In some situations our delocalised $\ell^2$-Betti numbers arise in Lott's setting. Recall that Lott is considering closed manifolds with a normal $G$-cover to which he associated his invariants. When the classifying space $\mathrm{B}G$ of a group $G$ has a homotopy type of a closed manifold we can consider its universal cover and it turns out that the associated delocalised $\ell^2$-Betti numbers agree.
		
		% Lott's delocalised $\ell^2$-Betti numbers involve the input of a closed Riemannian manifold and a normal $G$-cover. When the classifying space $\mathrm{B}G$ is a closed manifold, consider its universal cover, then Lott's delocalised $\ell^2$-Betti numbers for $\mathrm{B}G$ coincide with the above given formula for delocalised $\ell^2$-Betti numbers of $G$. 
		
		% In \cite{lott1999} the delocalised $\ell^2$-Betti numbers were shown to be metric independent for virtually nilpotent groups and hyperbolic groups. Note that the adjoint of the exterior derivative in the $p$-form Laplacian on the universal cover of a closed manifold involves metric of the manifold. An analogue of that in our setting may be that the delocalised $\ell^2$-Betti numbers are independent of the choice of the free resolution for the trivial module $\mathbb Z$ that we start with. 
		% \end{remark} 
	
	\begin{remark}
		Lott defined his invariants with respect to conjugacy class of non-trivial elements, which makes sense in view of the choice of name delocalised. 
		For convenience we include the case of the trivial element because all of our statements work uniformly. 
		
		% In his formulation, Lott defined delocalised $\ell^2$-Betti numbers only when $g\ne e$. This is because when $g$ is non-trivial, the values of some off-diagonal part of the heat operator will contribute to $\beta^{(2)}_{n, \langle g\rangle}(G)$. However, we include $g=e$ in our definition, and we shall regard the classical $\ell^2$-Betti numbers as a special case.  
	\end{remark}
	
	% The heat operator $e^{-t\Delta_n}$ belongs to a matrix algebra over $\Csr(G)$. But $\tau_{\langle g\rangle}$ does not necessarily extended  continuously to $\Csr(G)$. To circumvent this issue one many find a suitable smaller algebra containing $e^{-t\Delta_n}$ to which the delocalised traces extend continuously. See Theorem~\ref{thm:p_n in S} concerning this technical issue. 
	% L¹ does the job why not that?
	
	\subsection{Smooth subalgebras and K-theory pairing with delocalised traces} \label{subsec:smooth sal}
	
	For the purpose of this article a smooth subalgebra $\mathcal S$ of a $\Cs$-algebra is a dense Banach subalgebra which is closed under holomorphic functional calculus. 
    %For the purpose of this article, a smooth subalgebra $\mathcal{S}$ of a $C^*$-algebra $A$ is a dense subalgebra of $A$ that is closed under holomorphic functional calculus, equipped with a norm that makes it a Banach algebra, and such that the inclusion map $\mathcal{S} \hookrightarrow A$ is continuous. 
	An important feature of smooth subalgebras is that they have the same K-theory as the $\Cs$-algebra. See \cite[Appendix A]{bost}.
	For a given group $G$, we are interested in smooth subalgebras of its reduced group $\Cs$-algebra $\Csr(G)$. Let $g\in G$. 
	If delocalised traces extend continuously to a smooth subalgebra of $\Csr(G)$, this gives rise to a K-theory pairing by considering
	
	% If a smooth subalgebra $\mathcal S$ exists and the delocalised trace can be extended to it, then the delocalised traces on $\mathbb C G$ induce homomorphisms on the K-theory of $\Csr(G)$:
	\[
	\tau_{\langle g \rangle}: \mathrm{K}_0(\Csr(G))\cong \mathrm{K}_0(\mathcal S)\ra \mathbb C.
	\]
	  	%\begin{definition}
	%	A smooth subalgebra of a $\Cs$-algebra to which delocalised traces extend is called a Puschnigg algebra.
%--	Does not make sense as it must be said for groups, otherwise delocalised traces do  not make sense
	%\end{definition}
	
	The following result is a consequence of work of Puschnigg on traces on unconditional completions of group rings presented in Section 5 of \cite{puschnigg2010}. In particular it follows from the initial part of the proof of Theorem 5.2 and Proposition 5.5 in there.
	
	\begin{theorem} 
		\cite[Section 5]{puschnigg2010}
		Let $G$ be a hyperbolic group. There exists a smooth subalgebra of $\Csr(G)$ to which delocalised traces extend.
	\end{theorem}
    \begin{definition}
    	Let $G$ be a group. A smooth subalgebra of $\Csr(G)$ which contain the complex group ring $\mathbb CG$ and to which all delocalised traces extend is called a Puschnigg subalgebra.
    \end{definition}
    
 Recall that for a group $G$ of type $F_{n+1}$, its Laplacian $\Delta_n$ belongs to some matrices over $ \mathbb CG \subseteq \mathcal S$.
%	This subalgebra allows us to define delocalised traces in a cohomological way by pairing K-classes with the de
%	The next proposition clarifies that these delocalised traces agree with Lott's delocalised traces.
%	
%	heat kernel operator is needed.
%	
%	argue that 2.2 is 
%	cohomological delocalised $\ell^2$-Betti numbers
The next proposition demonstrates that the family of heat operators converges already in the Puschnigg subalgebras. The convergence is crucial for our later applications. Our proof is analogous to an argument given in the proof of \cite[Lemma 4]{lott1999}. 

 %We recall that for a given group $G$ a heat operator 
%$e^{-t \Delta_n}$ belongs to a matrix algebra over any Banach algebra completion of the group ring. In particular it belongs to a Puschnigg algebra.
%Fix $t$. The power series of the exponential is absolutely convergence;
%\[
%\sum_k {\|\frac{(t \Delta_n)^k}{k!}\|} \leq \sum _{k} {\frac {t^k \|\Delta_n\|^k}{k!}}\leq \sum_{k}{\frac {(tM)^k}{k!}}< \infty,
%\]
%where $M= \| \Delta_n\|$.
%In a Banach space, absolute convergence implies uniform convergence. In particular it implies that the heat operator belongs to $\ell ^1 G$ or Puschnigg algebra.

%	\begin{proposition}
%		The delocalised $\ell^2$-Betti number of Lott exists and agrees with the the delocalised trace pairing. 
%	\end{proposition}
%		Continuity: this heat Kernel limit is a convergence inside the Banach algebra  S
%		Need: $e^{-t\Delta_n} \to p_n$ is $\mathcal S$. Is it so?
\begin{remark}
    Hyperbolic groups are of type $F_{\infty}$. Therefore we will drop the assumption of type $F_{n+1}$ when stating our results concerning $p_n$, for $n\in \mathbb N$.
\end{remark}

        \begin{proposition}
        	Let $G$ be a hyperbolic group. If $\Delta_n$ has a spectral gap, then  
         $p_n\in \mathcal S$ and $$\lim_{t\to \infty} e^{-t \Delta_n}=p_n$$ 
        	in $\mathcal S$ for any Puschnigg subalgebra $\mathcal S \subseteq \Csr(G)$.
		\end{proposition}
	\begin{proof}
		The membership of the heat operator $e^{-t \Delta_n}$ in any Banach algebraic closure of the group ring $\mathbb CG$, and in particular to $\mathcal S$ is straightforward to verify. Furthermore, an easy application of holomorphic functional calculus implies that $p_n$ belongs to $\mathcal S$. 
		Due to the fact that $\mathcal S$ is holomorphically closed, the following computation can be done in $\mathcal S$.		
		Consider the decomposition
		\begin{align*}
			p_n - e^{-t \Delta_n} &= p_n (p_n - e^{-t \Delta_n}) + (id - p_n)(p_n - e^{-t \Delta_n})\\
			&= (id - p_n) e^{-t \Delta_n} (id - p_n)\\
			&= e^{(id - p_n)-t \Delta_n (id - p_n)}\\
			&:=  e^{-t \tilde{\Delta_n}}.
		\end{align*}
		We show that $\|e^{-t \tilde{\Delta_n}} \|$ goes to zero.
		Consider the family of one parameter semi groups $\{e^{-t \tilde{\Delta_n}}\}_{t>0}$ in the Banach algebra $\mathcal S$. By \cite [Theorem 1.22]{davies1980}, the limit below exists
		\[
			r = \lim_{t \to \infty} t^{-1} \ln \|e^{-t \tilde{\Delta_n}} \|,
		\]
		and we have that for all $t> 0$ the spectral radius of $e^{-t \tilde{\Delta_n}}$ is $e^{rt}$. Let $\lambda_0 > 0$ be the inf of the spectrum of $\tilde{\Delta_n}$. Then the spectral radius of $e^{-t \tilde{\Delta_n}}$ is $e ^{-t \lambda_0}$.
		Since $r < 0$, then $r/2 > r$, implying that for $t>t_0$
		\[
		t^{-1} \ln \|e^{-t \tilde{\Delta_n}} \| \leq r/2.
		\]
		When $t>1$ we obtain that for $C>0$
		\[
		\|e^{-t \tilde{\Delta_n}} \| \leq C e ^{-t \lambda_0/ 2}.
		\]
		Letting $t$ go to infinity, we obtain the desired convergence in $\mathcal S$.
	\end{proof}

	\begin{corollary}
		Let $G$ be a hyperbolic group. If $\Delta_n$ has spectral gap, then 
		the K-class $[p_n]$ belongs to $\mathrm{K}_0(\Csr(G))$.
	\end{corollary}
	
	\begin{proof}
		By hyperbolicity, the reduced group $\Csr(G)$ contains a Puschnigg subalgebra $\mathcal S$. In the presence of a spectral gap $p_n$ belongs to some matrix over $\mathcal S$. Consequently, the K-class $[p_n]$ belongs to $\mathrm K_0(\mathcal S) = \mathrm K_0(\Csr(G))$.		
	\end{proof}
	Assume that for a given group $G$ there is a Puschnigg subalgebra $\mathcal S \subseteq \Csr(G)$. The advantage of having $\lim_{t\to \infty} e^{-t \Delta_n}=p_n$ in $\mathrm M_{k_n}(\mathcal S)$ is that we can do K-theory pairings with delocalised traces.  %In the presence of a smooth subalgebra, e.g. when $G$ is hyperbolic or virtually nilpotent, 
    In such a situation we may rewrite (\ref{analytic def}) from Definition \ref{def: lott deloc betti numbers} to obtain an analogous formula to (\ref{def:l2 betti number definition}) for $\ell^2$-Betti numbers.

		\begin{definition} \label{def:deloc l2betti num formula}
			Let $G$ be a group of type $F_{n+1}$ and $g\in G$. Assume that there is a smooth subalgebra $\mathcal S \subseteq \Csr (G)$ to which the delocalised trace $\tau_{\langle g \rangle}$ extends. Assume further that the K-class of $p_n$ lies in $\mathrm K_0(\Csr (G))$.  We define the $n$-th delocalised $\ell^2$-Betti number of $G$ via the following pairing
			\begin{gather} \label{K-theoretic def}
				\beta ^{(2)}_{n, \langle g\rangle}(G) = \tau_{\langle g\rangle} ([p_n]).
			\end{gather}
		\end{definition}
  We refer to the above definition as the K-theoretic definition of delocalised $\ell^2$-Betti numbers. (See (\ref{analytic def}).)
		%Throughout this article Definition \ref{def:deloc l2betti num formula} as the one for the delocalised $\ell^2$-Betti numbers. 
		
		%\todo{should hyperbolicity + vitually nilpotency be pasrt of assumption in Def C? or just hyperbolicity? but later as some point cemmenet that for v.nil we have that too?}
		
		% \begin{remark}
			%     The formula \ref{def:deloc l2 betti} can be considered as the delocalised analog of the $\ell^2$-Betti numbers formula 
			%     \[
			%         \beta ^{(2)}_{n}(G) = \tau ([p_n]).
			%     \]
			%     Note that the canonical trace $\tau$ is defined on any amplification of the group von Neumann algebra $\mathrm{L}(G)$, in which without any further assumption higher Kazhdan projections $p_n$ live. 
			% \end{remark}
		% \end{remark}

	% Learnt that question do not belong to the prelimanry!
	
	% As observed an important ingredient to extend the above definition to a larger class of groups is to have smooth subalgebras of the kind we discussed. A natural question to ask is then the following.
	% \begin{question} For which other class of groups can one construct smooth subalgebras with the requirement that delocalised traces can continuously extend to?
		% \end{question}

	\section{Vanishing results for delocalised \texorpdfstring {$\ell^2$-} -Betti numbers} \label{sec: vanishing}
    In this section, we present our results concerning the vanishing of delocalised $\ell^2$-Betti numbers.  
    These results encompass connections to property (T) and $\ell^2$-Betti numbers. Furthermore, we provide additional evidence regarding a remark that Lott made concerning the vanishing of the invariant for torsion-free groups. 
  	% Moreover, we answer in affirmative one question by Lott concerning vanishing of the invariant for torsion-free groups.
    \begin{remark}
    The vanishing results in Proposition \ref{prop:l2 vanishing}, Corollary \ref{cor:property T}, and Proposition \ref{prop:vir.nilpotent} are obtained by appealing to our K-theoretic definition of  delocalised $\ell^2$-Betti numbers (Definition \ref{def:deloc l2betti num formula}). While in Proposition \ref{prop:l2 vanishing} and Corollary \ref{cor:property T} in the presence of a spectral gap for a relevant degree of Laplacian our K-theoretic definition and the analytic definition of delocalised $\ell^2$-Betti numbers agree, the situation with Proposition \ref{prop:vir.nilpotent} is different. For amenable groups a spectral gap is not expected, hence such groups do not fit in the setup of analytically defined delocalised $\ell^2$-Betti numbers. The key ingredient to the above mentioned results is that $p_n$ is zero.
    %Our results reflect the fact that  is zero.
    
    %Given the original definition (Definition \ref{def: lott deloc betti numbers}) the assumptions do not suffice to conclude them. Indeed we do not require a spectral gap. 
    %Despite this, because by the assumptions the K-class of $p_n$ is zero, so will be the pairing.
\end{remark}
    
	\begin{proposition} \label{prop:l2 vanishing}
		Let $G$ be a hyperbolic group. If $\beta_n^ {(2)} (G)$ vanishes, then $\beta_{n, \langle g \rangle}^ {(2)}(G)$ vanishes for all $g\in G$.
	\end{proposition}
	
	\begin{proof}
		If the $n$-th $\ell^2$-Betti number of $G$ vanishes, then by (\ref{def:l2 betti number definition}), $p_n$ vanishes. This leads to the vanishing of all pairings of $[p_n]$ with delocalised traces. Therefore, all $\beta_{n, \langle g \rangle}^ {(2)}(G)$ vanish.
	\end{proof}

 %\begin{remark}
     %Given Definition \ref{def: lott deloc betti numbers} without a spectral gap condition in Proposition \ref{prop:l2 vanishing} we would not be able to conclude that the invariant vanishes. However, appealing to Definition \ref{def:deloc l2betti num formula} this is possible because the projection is zero and hence the invariant vanishes. 
 %\end{remark}
	
	\begin{corollary}\label{cor:property T}
		Let $G$ be a hyperbolic property (T) group. Then $\beta^ {(2)}_{1, \langle{g}\rangle} (G)$ vanishes for all $g\in G$. 
	\end{corollary}
	
	\begin{proof}
		 Groups with property (T) have vanishing first $\ell^2$-cohomology \cite{delorme1977}. This implies that $\beta_1^{(2)}(G)$ vanishes.
		Proposition \ref{prop:l2 vanishing} can be now applied.
	\end{proof}

 	\begin{proposition} \label{prop:vir.nilpotent}
		Let $G$ be a virtually nilpotent group of type $F_{n+1}$. The delocalised $\ell^2$-Betti numbers $\beta^ {(2)}_{n, \langle{g}\rangle} (G)$ vanish for all $n\in \mathbb N$ and all $g\in G$.
	\end{proposition}  
 
	\begin{proof}
		 Virtually nilpotent groups have polynomial growth \cite{gromov81} and hence property RD \cite{jolissaint1990property}. 
         This implies that the smooth subalgebra 
          $L^2_{RD}(G) \subseteq \ell^1(G) \subseteq \Csr(G)$ from \cite{Jolissaint} is a Puschnigg subalgebra. 
          To see that the delocalised trace $\tau_{\langle g \rangle}$ extends continuously to $L_{RD}^2(G)$, note that because $G$ has polynomial growth, the size of the intersection of any conjugacy class $\langle g \rangle$ with the ball of radius $r$ grows at most polynomially in $r$. By definition, any function $f \in L_{RD}^2(G)$ decays faster than any polynomial with respect to the word length. Therefore, the sum
\[
\tau_{\langle g \rangle}(f) = \sum_{h \in \langle g \rangle} f(h)
\]
converges absolutely for all $f \in L_{RD}^2(G)$. This guarantees that $\tau_{\langle g \rangle}$ extends to a well-defined, continuous trace on $L_{RD}^2(G)$.
          Furthermore, infinite amenable groups have vanishing $\ell^2$-Betti numbers \cite{cheegergromov86}, their higher Kazhdan projections $p_n$ and hence their delocalised $\ell^2$-Betti numbers vanish.
	\end{proof}
 %---
%\todo{insert it somewhere better maybe?}
  %\begin{remark}
        % If $G$ is a FC-group, delocalised traces extend to the group von Neumann algebra $L(G)$. By definition higher Kazhdan projections live as well on matrices over $L(G)$. So in this case without requiring a spectral gap, we can do the trace pairing and compute the invariant.
         %\todo{perhaps not, as in our definition we require smooth sub algebra of the reduced group C*-algebra!}
     %\end{remark}
    % ---
	%\begin{remark} 
		%A direct implication of Proposition \ref{prop:vir.nilpotent} is that free abelian groups have vanishing delocalised $\ell^2$-Betti numbers. This offers an alternative approach to establish Proposition 2 in \cite{lott1999}. In that work it is shown that for a manifold $M$ with a free abelian fundamental group $\Gamma$, and for $g\in \Gamma$ with finite conjugacy class, the deloclaised $\ell^ 2$-Betti number $\beta^2_{n, \langle g \rangle} (M)$ vanish. 
		
		% In fact, he considered the $n$-torus $\mathbb T^n$ and its universal covering whose fundamental group is $\mathbb Z^n$. 
		% The delocalised $\ell^2$-Betti numbers can be expressed using the ingredient of the heat kernel of the $n$-th Laplacian, which is a $\mathbb Z^n$-invariant elliptic operator. Employing Fourier analysis and the spectrum of a family of elliptic operators, the calculation can be reduced to an integral over the Pontryagin dual $\mathbb T^n$ of $\mathbb Z^n$. See Proposition 2 in \cite{lott1999}.  
	%\end{remark}
	
	The following proposition can be compared with \cite[Proposition 9]{lott1999} which deals with closed oriented hyperbolic manifolds.
	
	\begin{proposition}
		\label{prop:torfreehyper}
		Let $G$ be a torsion-free hyperbolic group. 
		Assume that $\Delta_n$ has a spectral gap.
		Then the $n$-th delocalised $\ell^2$-Betti numbers vanish for all non-trivial $g\in G$.
	\end{proposition}

\begin{proof}
Since $G$ is a hyperbolic group, the Baum-Connes conjecture holds by the work of Lafforgue \cite{lafforgue2002}. Thus, the assembly map $\mu: \mathrm K_0^G(\underline{E}G) \to \mathrm K_0(\Csr(G))$ is an isomorphism. Because $\Delta_n$ has a spectral gap and $G$ is hyperbolic, the K-class of the higher Kazhdan projection $[p_n]$ belongs to $\mathrm K_0(\Csr(G))$ (cf. Corollary 2.10). By the surjectivity of the assembly map, there exists an equivariant K-homology class $x \in \mathrm K_0^G(\underline{E}G)$ such that $\mu_0(x) = [p_n]$. 

By the definition of the equivariant K-homology of $\underline{E}G$ as a direct limit, there exist a $G$-finite, proper $G$-CW-complex $X$, an element $y\in \mathrm K_0^G(X)$ given by a Fredholm module, and a $G$-equivariant map $f: X\rightarrow \underline{E}G$ such that the induced map $f_*: \mathrm K_0^G(X)\rightarrow \mathrm K_0^G(\underline{E}G)$ sends $y$ to $x$. By definition, $\mu_0(f_*(y))$ is the higher index of $y$, denoted $\mathrm{ind}_G(y)$. Thus, $\mathrm{ind}_G(y)=[p_n]$.  

Furthermore, it was shown in \cite{BaumHigsonSchick2010} that for spaces with proper and cocompact discrete group actions, the analytic K-homology $\mathrm K_0^G(X)$ is naturally isomorphic to the Baum-Douglas geometric $K$-homology. Under this isomorphism, $y$ can be represented by a geometric cycle $(M, E, \phi)$ such that $y=\phi_*[D_E]$. Here, $M$ is a complete $\mathrm{Spin}^c$-manifold equipped with a proper, cocompact, and isometric action of $G$; $E$ is a $\mathbb{Z}_2$-graded $G$-equivariant Hermitian Clifford module over $M$; $\phi: M\rightarrow X$ is a $G$-equivariant continuous map; and $D_E$ is the $G$-equivariant elliptic differential operator of Dirac type associated to the $\mathrm{Spin}^c$-structure and the twisting bundle $E$. Consequently, the image of this geometric cycle under the assembly map is the higher index of the operator, i.e. $\mathrm{ind}_G(D_E)=\mathrm{ind}_G(y) = [p_n] \in \mathrm K_0(\Csr(G))$. 

We now apply the $L^2$-Lefschetz fixed-point formula to evaluate the delocalised trace. In our situation, the assumptions of the localised index theorem in \cite{WangWang16} are automatically satisfied: $M$ is a proper cocompact $G$-manifold, and the standard finite propagation speed and heat kernel asymptotics for the Dirac-type operator $D$ on $M$ ensure that the heat operators $e^{-tD^2}$ are of $\langle g \rangle$-trace class \cite[Section 3]{WangWang16}. The compactness of $M/G$ is essential to the finite $\langle g \rangle$-trace estimate. 

By \cite[Theorem 6.1]{WangWang16}, the delocalised trace of the higher index is given by a cohomological formula over the fixed-point submanifold $M^g$:
\begin{equation*}
\tau_{\langle g \rangle}(\mathrm{ind}_G(D_E)) = \int_{M^g / Z_G(g)} A_S(g)(x) \, dx,
\end{equation*}
where $A_S(g)(x)$ is the local Lefschetz integrand associated to the element $g$, given by characteristic forms restricted to the fixed-point submanifold $M^g$. 
Because the action of $G$ on $M$ is proper, all isotropy subgroups must be finite. 
Note that the in order to apply $\tau_{\langle g \rangle}$ to $\mathrm{ind}_G(D_E)\in \mathrm K_0(\Csr(G))$ we required hyperbolicity of $G$. 
However, since $G$ is torsion-free by hypothesis, the isotropy subgroup for $g\neq e$ is trivial. This implies that the action of $G$ on $M$ is free. 
Consequently, for any non-trivial element $g \in G$ (i.e., $g \neq e$), the fixed-point submanifold $M^g$ is empty. The integral over an empty set vanishes, yielding:
\begin{equation*}
\beta_{n,\langle g\rangle}^{(2)}(G) = \tau_{\langle g \rangle}([p_n]) = \tau_{\langle g \rangle}(\text{ind}_G(D_E)) = 0.
\end{equation*}
This concludes the proof.
\end{proof}

%\begin{proof}
%The Baum-Connes conjecture holds for hyperbolic groups \cite{lafforgue2002}. Hence, for $[p_n]\in \mathrm K_0(\Csr(G))$, there exists an abstract elliptic operator whose higher index is $[p_n]$. Moreover, given that $G$ is discrete, one can find via geometric cycles \cite{BaumHigsonSchick2010} a representative of the abstract elliptic operator that is $G$-equivariant elliptic operator $D$ lives on a complete manifold $M$ carrying a proper cocompact $G$-action.  
%\\		
%By the $L^2$-Lefschetz fixed point formula, the delocalised trace of the higher index $\tau_{\langle g \rangle}(\mathrm{ind}_G(D))$
%has a cohomological formula over the fixed point submanifold $M^g$ (refer to Definition 5.1 and Theorem 6.1 of \cite{WangWang16}.)
%The properness of the action implies that all isotropy subgroups are finite. 
%However, the fact that $G$ is torsion-free implies that the isotropy subgroups are trivial. This indicates that the action is free, and thus $M^g$ is empty since $g\neq e$. Consequently, the $L^2$-Lefschetz number vanishes, resulting  in 
%\[
%\beta_{n, \langle g\rangle}^{(2)}(G)=\tau_{\langle g\rangle}([p_n])=\tau_{\langle g \rangle}(\mathrm{ind}_G(D))=0.
%\]
%\end{proof}
	The previous proposition confirms Lott's speculation concerning vanishing of delocalised $\ell^2$-Betti numbers for discrete torsion-free groups. See \cite[page 6]{lott1999}. 
	
	\begin{remark}
		In the proof of Proposition~\ref{prop:torfreehyper}, we used vanishing of the $L^2$-Lefschetz number $\chi_g(G)$. It turns out that the non-vanishing of $\chi_g(G)$ for the Euler operator is an obstruction to the vanishing of the delocalised $\ell^2$-Betti numbers. Indeed, we have 
		\[
		\chi_g(G)=\sum_{n\ge 0}(-1)^n\beta^{(2)}_{n, \langle g\rangle}(G). 
		\]  
		If $\chi_g(G)$ is not vanishing, then there exists an $n$ such that $\beta^{(2)}_{n, \langle g\rangle}(G)\neq 0$.
		This generalises Proposition 11 on page 29 of \cite{lott1999} from the context of finite groups to that of  discrete groups with proper cocompact action. 
	\end{remark}
		%\begin{itemize}
	%       \item Similar to Kazhdan projections can we show that higher Kazhdan projections are central? No, $\Csr(F_n)$ is simple. If higher Kazhdan projection was central, then one would get an ideal, which is not possible.
	%\end{itemize}
	
	\section{Computations of higher Kazhdan projections and \mbox{delocalised}~ \texorpdfstring {$\ell^2$-} -Betti numbers}
	\label{sec: computations}
	
	In this section we describe the K-theory class of higher Kazhdan projections $p_n$ for $n>0$ in specific classes of groups. The results are obtained through the identification of the kernel of $\Delta_n$. As an application, we compute the delocalised $\ell^2$-Betti numbers. These establish the first non-vanishing results for infinite groups.
	
	\begin{lemma} \label{lem:existence}
		Let $G$ be a virtually free group. Then $p_1$ lies in matrices over $\Csr(G)$, its K-class belongs to
		$\mathrm{K_0}(\Csr(G))$ and it is non-zero. All other $p_n$'s vanish when $n\neq 1$.
	\end{lemma}
	
	\begin{proof}
		Assume that some free group $\mathbb F_k$ is a finite index subgroup of $G$. To establish the presence of a spectral gap, we rely on Lemma \ref{lem:spectral gap condition}. The non-amenability of $G$ implies that $\mathrm H^1(G, \ell^2 (G))$ is reduced, and  
		Shapiro lemma implies that $\mathrm H^2(G, \ell^2(G))$ is reduced (even vanishes) as it is so for $\mathbb F_k$. Consequently, $p_1$ belongs to matrices over $\Csr(G)$. 
		Recall equation \ref{def:l2 betti number definition}
		$\tau([p_1]) = \beta^{(2)}_{1}(G)$. Since $\beta^{(2)}_{1}(G)$ is non-zero, it follows that $[p_1]$ is non-zero. Additionally, given the faithfulness of $\tau$, we note that $\beta^{(2)}_{n}(G)$ vanishes for $n> 1$, leading to the vanishing of $p_n$.
	\end{proof}
	
	\subsection{Free product of finite cyclic groups}
	Consider the free product $G= \mathbb Z_m * \mathbb Z_n$ for $m \geq 2$ and $n \geq 3$. These groups are known to be virtually free. The goal is to find representatives for the K-class $[p_1]$ associated to these groups. As a byproduct, we compute the delocalised $\ell^2$-Betti numbers of $G$.
	
	\begin{lemma}
		\label{lem:Delta1}
		Let $G= \mathbb Z_m * \mathbb Z_n$ for $m \geq 2$ and $n \geq 3$. The first Laplacian $\Delta_1$ for $G$ is
		\[
		\Delta_1=d_0d_0^*+d_1^*d_1=
		\begin{bmatrix}2-s-s^{-1}+m\sum_{0 \leq i < m}s^i & (1-s^{-1})(1-t)\\
			(1-t^{-1})(1-s) & 2-t-t^2+n\sum_{0 \leq j < n}t^j
		\end{bmatrix},
		\]
		where $s$ and $t$ are generators for $\mathbb Z_m$ and $\mathbb Z_n$, respectively. 
	\end{lemma}
	
	\begin{proof}
		In order to describe the Laplacian $\Delta_1$ we need a suitable cochain complex. For that, we take the free resolution of the trivial $\mathbb ZG$-module $\mathbb Z$ associated with the presentation complex of $G$, which is
		\begin{equation*}
			\cdots 
			\longrightarrow 
			\mathbb Z G^{\oplus 2} 
			\stackrel{\delta_1}\longrightarrow 
			\mathbb Z G ^ {\oplus  2} 
			\stackrel{\delta_0}\longrightarrow 
			\mathbb ZG 
			\stackrel{\epsilon}\longrightarrow 
			\mathbb Z 
			\longrightarrow
			0 ,
		\end{equation*}
		where the connecting maps are defined as follows:
		%\begin{align*}
			%&\epsilon(g) = 1 \quad g\in G,&\\
			%&\delta_0(1,0) = 1-s \quad s\in \mathbb Z_m, &
			%\delta_0(0,1)=1-t \quad t\in \mathbb Z_n,\\
			%&\delta_1(1,0) = \bigl(\sum_{0 \leq i < m}s^i, 0\bigr), & 
			%\delta_1(0,1)=\bigl(0, \sum_{0 \leq i < n}t^j\bigr).
		%\end{align*}

\[
\begin{aligned}
\epsilon(g) &= 1 \quad g\in G, \\[6pt]
\delta_0(1,0) &= 1-s \quad s\in \mathbb Z_m ,
&\qquad
\delta_0(0,1) &= 1-t \quad t\in \mathbb Z_n,\\[4pt]
\delta_1(1,0) &= \bigl(\sum_{0 \le i < m} s^i, 0\bigr),
&\qquad
\delta_1(0,1) &= \bigl(0, \sum_{0 \le j < n} t^j\bigr).
\end{aligned}
\]
       
		Applying the functor $\mathrm {Hom}_{\mathbb Z G}( \,\cdot \,, \ell^ 2(G))$ and identifying the terms with direct sums of $\ell^2(G)$, yield
		\begin{equation*}
			\ell^ 2 (G) 
			\stackrel{d_0}\longrightarrow 
			\ell^ 2(G)^ {\oplus  2} 
			\stackrel{d_1}\longrightarrow
			\ell^ 2(G)^ {\oplus 2}
			\longrightarrow 
			\cdots,
		\end{equation*}
		where the connecting maps can be identified with the operators
		\begin{gather*}
			d_0 = \begin{bmatrix}
				1-s & 1-t
			\end{bmatrix}, \qquad 
			d_1 = \begin{bmatrix}
				\sum_{0 \leq i < m} s^i & 0 
				\\
				0   & \sum_{0 \leq j < n} t^j
			\end{bmatrix}.
		\end{gather*}
		Taking the transpose and computing $\Delta_1 = d_0d_0^*+d_1^*d_1$, we obtain the desired operator.
	\end{proof}
	
	\noindent The concrete $\Delta_1$ from Lemma~\ref{lem:Delta1} is an essential ingredient to the subsequent main theorem.
	
	\begin{theorem} \label{thm: Z_n*Z_m}
		Let $G= \mathbb Z_m \ast \mathbb Z_n$ with $m \geq 2$ and $n \geq 3$. The K-class of the higher Kazhdan projection $p_1$ in $\mathrm K_0(\Csr(G))$ can be represented by
		\[
		[p_1] = [1] - \left[\frac{1}{m}
		\sum_{0 \leq i < m}{s^i}\right] - 
		\left[\frac{1}{n}
		\sum_{0 \leq j <n}{t^j}\right].
		\]
		% This in particular applies to $\mathrm{PSL(2, \mathbb Z)} = \mathbb Z_2 * \mathbb Z_3$.
	\end{theorem}
	\begin{remark}
        In a recent preprint we describe the K-theory class $[p_1]$
      for non-amenable virtually free groups, including free products $\mathbb Z_m \ast \mathbb Z_n$.
 See \cite[Corollaries 6.3 and 6.5]{pooyarenwang2025}. The formula obtained there expresses $[p_1]$
as an alternating sum of averaging projections associated to the isotropy subgroups appearing in a fundamental domain for the action of $G$ on its Bass–Serre tree:
       \begin{equation*}
        [p_1]=\sum_{v\in \, \mathrm{edge}\ (G \backslash X)}{[p_{v}]}\,\,\,\,\,-\sum_{e\in \, \mathrm{vert}\  (G \backslash X)}{[p_{e}]}\in \mathrm{K_0(C^*_{red}}(G)).     
    \end{equation*}    
 In the specific case  
        $G= \mathbb Z_m \ast \mathbb Z_n$ the fundamental domain consists of a single edge with trivial isotropy and two vertices with isotropy 
        $\mathbb Z_m$ and $\mathbb Z_n$. Thus the description of $[p_1]$ given in Theorem~\ref{thm: Z_n*Z_m} agrees with the formula of \cite[Corollary 6.5]{pooyarenwang2025}. 
        More generally, the framework developed in \cite{pooyarenwang2025} is closely related to the combinatorial Euler characteristic 
$Eul^{cmb}(G)$
 introduced by Emerson and Meyer \cite{EM}. We show in \cite[Theorem 4.1]{pooyarenwang2025} that if $p_n$
 is the only non-vanishing higher Kazhdan projection of $G$, then its K-theory class can be expressed as an alternating sum of averaging projections associated with the finite isotropy occurring in a suitable
$G$-finite model of $\underline EG$.   
\end{remark}
    
	To describe the K-class of Kazhdan projection we will solve the equation $\Delta_1x=0$, for $x\in \ell^2 (G)^{\oplus 2}$. This requires some preparatory lemmas. Before that we fix some notations.
	
	%\noindent \textbf{Notation}. 
	%Let $p$ and $q$ be the averaging projections associated with the trivial representations of $\mathbb Z_m$ and $\mathbb Z_n$, respectively, i.e.
	%\begin{equation}
	%	\label{eq:pq}
	%	p=\frac{1}{m}{\sum_{0 \leq i < m}{s^i}} \quad \text{and} \quad 
	%	q=\frac{1} {n} {\sum_{0 \leq j <n}{t^j}}.
	%\end{equation}

\noindent \textbf{Notation.} Let $p$ and $q$ be the averaging projections of the finite subgroups $\mathbb{Z}_m$ and $\mathbb{Z}_n$ in the group algebra. When acting on the left regular representation, they project onto the $\mathbb{Z}_m$- and $\mathbb{Z}_n$-invariant subspaces, respectively, i.e.,
\begin{equation}
    \label{eq:pq}
    p=\frac{1}{m}{\sum_{0 \leq i < m}{s^i}} \quad \text{and} \quad 
    q=\frac{1} {n} {\sum_{0 \leq j <n}{t^j}}.
\end{equation}

	Moreover, let 
	\[
	C=\begin{bmatrix}p & 0 \\ 0 & q\end{bmatrix}
	\]
	be the projection associated with $p$ and $q$.
	\\
	Further, for a Hilbert space $H$ we denote by $P_{H}$ its associated projection.
	
	The following lemma characterises the $l^2$-solutions of $\Delta_1 x = 0$.
	
	\begin{lemma}
		\label{lem:SolDelta1}
		Let $G= \mathbb Z_m \ast \mathbb Z_n$ with $m \geq 2$ and $n \geq 3$. 
		Let $x\in \ell^2 (G)^{\oplus 2}$. Then 
		$\Delta_1 x =0$ 
		if and only if there are 
		$z\in \ell^2 (G)^{\oplus 2}$ and 
		$a\in \mathrm{im}(1-p)\cap\mathrm{im}(1-q)$
		such that 
		\[
		  x=(I - C)z  \qquad \text{and} \qquad		
		  \begin{bmatrix}1-s & 0 \\ 0 & {1-t}\end{bmatrix}z= \frac{1}{\sqrt 2}\begin{bmatrix}a\\ -a\end{bmatrix}. 
		\] 
		% for some $a\in \mathrm{im}(1-p)\cap\mathrm{im}(1-q).$
	\end{lemma}
	
	\begin{proof}
		The first Laplacian $\Delta_1$ explicitly described in \ref{lem:Delta1} can be rewritten as 
		\[
			\Delta_1 = 
			\begin{bmatrix}m^2p & 0\\0 & n^2q 
			\end{bmatrix}
			+
			\begin{bmatrix} {1-s^{-1}} & 0 \\ 0 & {1-t^{-1}}\end{bmatrix}\begin{bmatrix}1 & 1 \\ 1 & 1\end{bmatrix}
			\begin{bmatrix}{1-s} & 0 \\ 0 & {1-t}\end{bmatrix}.
		\]
		Suppose $x\in\ker \Delta_1$. By applying $C$ to $\Delta_1x=0$, we find that 
		\[
			\begin{bmatrix}
			m^2 & 0\\
			0 & n^2
			\end{bmatrix} Cx =
			\begin{bmatrix}
				m^2p & 0\\0 & n^2q 
		\end{bmatrix} x=0.
		\]
		It follows that $x\in \ker C=\mathrm{im}(1-C)$.
		Hence, there exists $z \in \ell^2(G) ^{\oplus 2}$ such that $x=(1-C)z$. 
		Considering $1-s=(1-p)(1-s)$ and $1-t=(1-q)(1-t)$, we have 
		\begin{equation}
			\label{eq2}
			\begin{bmatrix} {1-s^{-1}} & 0 \\ 0 & {1-t^{-1}}\end{bmatrix}\begin{bmatrix}1 & 1 \\ 1 & 1\end{bmatrix}\begin{bmatrix}{1-s} & 0 \\ 0 & 1-t \end{bmatrix}z=0.
		\end{equation}
		% Since $\begin{bmatrix}1 & 1\ \\ 1 & 1\end{bmatrix}$ is symmetric, 
		% there exists a 
		The unitary $U=\frac{1}{\sqrt 2}\begin{bmatrix} 1& 1\\ 1 & -1\end{bmatrix}$ satisfies  
		\[
		U\begin{bmatrix}2 & 0 \\ 0 & 0\end{bmatrix} U^*
		=
		\begin{bmatrix}1 & 1 \\ 1 & 1\end{bmatrix}.
		\]
		Thus the equation (\ref{eq2}) can be rewritten as 
		$A^*Az=0$, where 
		$
		A =
		\begin{bmatrix}\sqrt 2 & 0 \\ 0 & 0\end{bmatrix}U^*\begin{bmatrix}1-s & 0 \\ 0 & 1-t\end{bmatrix}
		$.
		Furthermore, $A^*Az=0$ is equivalent to 
		% $z^*A^*Az=0$, which is in turn equivalent to 
		$Az=0$. 
		Hence, we only need to solve the following equation
		\[
		U^*
		\begin{bmatrix} {1-s} & 0 \\ 0 & {1-t} 
		\end{bmatrix}\begin{bmatrix} z_1 \\ z_2 \end{bmatrix}
		=
		\begin{bmatrix} 0 \\ a
		\end{bmatrix},
		\]
		for some $a \in l^2(G)$.
		Applying $U$ to the left hand side, we infer that $\Delta_1x=0$ implies that there are $z\in \ell^2(G)^{\oplus 2}$, $x=(I-C)z$ and $a\in\mathrm{im}(1-p)\cap\mathrm{im}(1-q)$ satisfying
		\[
			\begin{bmatrix}1-s & 0 \\ 0 & {1-t}\end{bmatrix}z
			=
			\frac{1}{\sqrt 2}\begin{bmatrix}a\\-a
			\end{bmatrix}. 
		\] 
		This proves the first part of the statement.
		The converse is evident through straightforward computation.
		\end{proof}

	For later reference we consider the following factorization of $1-p$ and $1-q$:
	\begin{equation}
		\label{eq3}
		1-p = \frac{1}{m}[(1-s)+\cdots+(1-s^{m-1})] =
		k(1-s)=(1-s)k,
	\end{equation}
	where $k:=\frac1{m}[1+(1+s)+\cdots+(1+s+\cdots+s^{m-2})]\in\mathbb C G$.
	Similarly, there exists $l:=\frac1{n}[1+(1+t)+\cdots+(1+t+\cdots+t^{n-2})]\in\mathbb C G$ such that 
	\begin{equation}
		\label{eq4}
		1-q = l(1-t)=(1-t)l.
	\end{equation}
		
	\begin{remark}\label{rem:identification}
	Consider the canonical inclusion 
	$\mathrm M_n(\Csr(G))\subset\mathcal{B}(\ell^2(G)^{\oplus n})$ induced by the left regular representation of the group. 
	From this perspective, every projection $r \in \mathrm M_n(\Csr(G))$ gives rise to an orthogonal projection $R \in \mathcal B(\ell^2(G)^{\oplus n})$.
	%\[
		%R: H^n\rightarrow H^n \qquad h\mapsto pH
	%\] 
	%onto the subspace $rH^n$
	We identify $r$ with $R$ and  
	shall write $[r\ell^2(G)^{\oplus n}]$ or $[R\ell^2(G)^{\oplus n}]$ to denote the K-class of the projection from $\ell^2(G)^{\oplus n}$ to $r\ell^2(G)^{\oplus n}$. We identify then $[r\ell^2(G)^{\oplus n}]=[R\ell^2(G)^{\oplus n}]=[r]\in \mathrm K_0(A)$.
	\end{remark}
 
 	We proceed to find representatives for the K-theory class of the projection associated to $\ker \Delta_1$.	  
 	
	\begin{lemma}
		\label{lem:projker}
		Let $k$ and $l$ be as in (\ref{eq3}) and (\ref{eq4}). The $\ell^2$-kernel of $\Delta_1$ is   
		\[
		\ker\Delta_1
		=
		\left\{\begin{bmatrix}ka\\ -la\end{bmatrix} \mid  a\in \mathrm{im}(1-p)\cap\mathrm{im}(1-q)\right\}.
		\] 
		Moreover, the projection $P_{\ker \Delta_1}$ has the same $K$-class as the projection $P_{H_1}$ onto the Hilbert subspace 
		\begin{equation}
			\label{eq:H1}
			H_1:=\left\{\begin{bmatrix}a\\ -a\end{bmatrix} \mid  a\in \mathrm{im}(1-p)\cap\mathrm{im}(1-q)\right\}
			\subset \ell^2(G)\oplus \ell^2(G).
		\end{equation}
		In other words,  
		\[
			[P_{\ker\Delta_1}]=[P_{H_1}]\in \mathrm K_0(\Csr(G)).
		\]
	\end{lemma}
	
	\begin{proof}
		%Let $H=l^2(G).$
		In view of (\ref{eq3}), (\ref{eq4}) and Lemma \ref{lem:SolDelta1}, we have that $\Delta_1 x=0$ if and only if $x$ satisfies  
		\[
			x
			=
			\begin{bmatrix}1-p & 0 \\ 0 & 1-q\end{bmatrix}z
			= 
			\begin{bmatrix}k & 0 \\ 0 & l\end{bmatrix}
			\begin{bmatrix}1-s & 0 \\ 0 & 1-t \end{bmatrix}
			z
			=
			\begin{bmatrix}k & 0 \\ 0 & l\end{bmatrix}
			\begin{bmatrix}\frac{1}{\sqrt 2}a \\-\frac{1}{\sqrt 2}a
			\end{bmatrix} 
		\]
		for $z=(I-C)x$ and some $a\in \ell^2(G)$ satisfying $a\in\mathrm{im}(1-p)\cap\mathrm{im}(1-q)$.
		In other words, we have
		\begin{align*}
			\ker\Delta_1
			=
			&
			\left\{ 	
				\begin{bmatrix}1-p & 0 \\ 0 & 1-q\end{bmatrix}z
				~~\mid~~  
				\begin{bmatrix}1-s & 0 
				\\ 
				0 & 1-t\end{bmatrix}z=\begin{bmatrix}\frac{1}{\sqrt 2}a \\-\frac{1}{\sqrt 2}a
			\end{bmatrix}, a\in\mathrm{im}(1-p)\cap\mathrm{im}(1-q)  	
		   \right\}
		   \\
			=
			&
			  \left\{ 
			  	\begin{bmatrix}\frac{ka}{\sqrt 2} \\ \frac{-la}{\sqrt 2}\end{bmatrix}
			 	 ~~\mid~~ 
			  	a\in\mathrm{im}(1-p)\cap\mathrm{im}(1-q)  
			  \right\}.
		\end{align*}
	%	Due to the factorization in (\ref{eq3}) and (\ref{eq4}), $k$ and $1-s$ are inverse to each other in $(1-p)\ell^2(G)$ and $l$ and $1-t$ are so in $(1-q)\ell^2(G)$. 
	%With respect to the decomposition 
	%$\ell^2(G) = p\ell^2(G)\oplus(1-p)\ell^2(G)$, and together with the fact that 
	%$1-s=(1-p)(1-s)$ and $1-t=(1-q)(1-t)$
	%the operator $v_1:=p+(1-s)$	is invertible with the inverse given by $v_1^{-1}:= 1+k$.
	%The same holds for the operator $v_2:=q+(1-t)$. Its inverse is given by $v_2^{-1}:= 1+l$.
Due to the factorizations in (\ref{eq3}) and (\ref{eq4}), the operators $k$ and $1-s$ act as mutual inverses on the subspace $(1-p) \ell^2(G)$, just as $l$ and $1-t$ act as mutual inverses on $(1-q) \ell^2(G)$. To see this, note that $k(1-p)=(1-p)k$, and since $1-s=(1-p)(1-s)$, the space $(1-p)\ell^2(G)$ is invariant under the action by $1-s$. With respect to the orthogonal decomposition $\ell^2(G)=p \ell^2(G) \oplus(1-p) \ell^2(G)$, the operator $v_1:=p+(1-s)$ respects this block-diagonal structure. It acts as the identity on $p \ell^2(G)$ and as $1-s$ on $(1-p) \ell^2(G)$. Therefore, $v_1$ is invertible with its inverse given by $v_1^{-1}:= p + k(1-p)$. 
Indeed, a direct calculation verifies this (noting that $sp = p$):
\begin{align*}
[p+(1-s)][p+k(1-p)] &= p^2 + pk(1-p) + (1-s)p + (1-s)k(1-p) \\
&= p + 0 + (p-sp) + (1-p) \\
&= p + 0 + 0 + (1-p) \\
&= 1.
\end{align*}

Consider the invertible operator $V:=\begin{bmatrix} v_1 & 0 \\ 0 & v_2 \end{bmatrix}$ in $M_2(\Csr(G))$. Let us compute the image of $\ker \Delta_1$ under $V$. For $a \in \mathrm{im}(1-p)$, we have $pa = 0$. This implies that $pka = pk(1-p)a = 0$. Consequently, the action of $v_1$ on $ka$ yields:
\[
v_1(ka) = (p+1-s)ka = pka + (1-s)ka = 0 + (1-p)a = a.
\]
Applying the same logic to the second coordinate, we find:
\begin{align*}
V(\ker \Delta_1) &= \left\{ \begin{bmatrix} v_1 ka \\ -v_2 la \end{bmatrix} \;\middle|\; a \in \mathrm{im}(1-p) \cap \mathrm{im}(1-q) \right\} \\
&= \left\{ \begin{bmatrix} a \\ -a \end{bmatrix} \;\middle|\; a \in \mathrm{im}(1-p) \cap \mathrm{im}(1-q) \right\} = H_1.
\end{align*}

To conclude the statement in K-theory, consider the element $e := V^{-1} P_{H_1} V \in M_2(\Csr(G))$. Since $P_{H_1}$ is an orthogonal projection onto $H_1$, $e$ is an idempotent operator whose image is precisely $V^{-1}(H_1) = \ker \Delta_1$. In a $\Cs$-algebra, any idempotent defines the same $\mathrm K_0$-class as the orthogonal projection onto its image, so $[P_{\ker \Delta_1}] = [e] \in \mathrm K_0(\Csr(G))$. Furthermore, since $e$ and $P_{H_1}$ are conjugate idempotents, they share the same K-theory class. Thus, we obtain:
\[
[P_{\ker \Delta_1}] = [e] = [V^{-1} P_{H_1} V] = [P_{H_1}].
\]
This finishes the proof.
		% originaly writen by Hang
		%We then have
		%\begin{equation*}
			%V(\ker \Delta_1)
			%=
		%	\ker(V \Delta_1 V^{-1}) 
		%	\quad \text{and} \quad 
		%	P_{\ker\Delta_1}
		%	=
		%	V^{-1}P_{\ker(V\Delta_1 V^{-1})}V.
		%\end{equation*}
		%This implies that as $K$-theory elements
		%\[
		%[P_{ker\Delta_1}]=[P_{\ker(V\Delta_1V^{-1})}],
		%\]
	%	where the kernel of ${(V\Delta_1V^{-1})}$ satisfies
		%\[
		%\ker(V\Delta_1 V^{-1})
		%=
		%V(\ker \Delta_1) 
		%= 
	%\]
		%which is isomorphic to $\left\{\begin{bmatrix}a\\ -a\end{bmatrix} \mid  a\in \mathrm{im}(1-p)\cap\mathrm{im}(1-q)\right\}$
		%as desired.
	\end{proof}
	
	Before presenting our proof for Theorem \ref{thm: Z_n*Z_m} we need two more technical lemmas. 
	Recall the Hilbert space $H_1$ defined in (\ref{eq:H1}). We consider four more Hilbert subspaces of $\ell^2(G)\oplus \ell^2(G)$.
	\begin{align} \label{eq7}
		& H_2:=\left\{\begin{bmatrix}c\\ -c\end{bmatrix} \mid  c\in \mathrm{im}p+\mathrm{im}q \right\}  \quad
		\tilde H_2:=\left\{\begin{bmatrix}pw\\ qv\end{bmatrix} \mid  v,w\in \ell^2(G) \right\}
		\\
		& \quad H_3:=\left\{\begin{bmatrix}b \\ b \end{bmatrix} \mid  b\in \ell^2(G) \right\}   
		\quad \tilde H_3:=\left\{\begin{bmatrix}(1-p)u \\ (1-q)u \end{bmatrix} \mid  u \in \ell^2(G) \right\}
	\end{align}
	The next lemma provides two decompositions of $\ell^2(G)\oplus \ell^2(G)$.
	
	\begin{lemma}
		\label{lem:directsum}
		There are orthogonal direct sums of Hilbert spaces
		\[
			\ell^2(G)\oplus \ell^2(G) = H_1\oplus H_2\oplus H_3\quad \text{and} \quad
			\ell^2(G)\oplus \ell^2(G) = H_1\oplus \tilde H_2\oplus \tilde H_3.
		\]
	\end{lemma}
	
	\begin{proof}
		The identity $$(\mathrm{im}(1-p)\cap\mathrm{im}(1-q))^{\perp}=\mathrm{im}(1-p)^{\perp}+\mathrm{im}(1-q)^{\perp}=\mathrm{im}(p)+\mathrm{im}(q),$$
		implies that 
		\begin{equation}
			\label{eq:yeah}
			\ell^2(G)=(\mathrm{im}(1-p)\cap\mathrm{im}(1-q))+\mathrm{im}(p)+\mathrm{im}(q).
		\end{equation}
		Applying this, it is then straightforward to verify that $\ell^2(G)\oplus \ell^2(G)=H_1\oplus H_2\oplus H_3$.
		
		%Then we show that there is a direct sum (but not orthogonal sum) of Hilbert spaces
		%\[
		%	\ell^2(G)\oplus \ell^2(G)=H_1\oplus \tilde H_2\oplus H_3
		%\]
		Now we turn to the second decomposition. It is easy to check that $H_1$ and $\tilde H_2$, $H_1$ and $H_3$ are orthogonal.
		Using $\mathrm{im}(p)\cap\mathrm{im}(q)=\{0\}$, we infer that $\tilde H_2\cap H_3=\{0\}$.
		Next we show that their sum is $\ell^2(G)\oplus \ell^2(G)$.
		From (\ref{eq:yeah}) it follows that 
		\begin{equation}
			\label{eq:H}
			\ell^2(G)=\{2a+pw-qv \mid a\in\mathrm{im}(1-p)\cap\mathrm{im}(1-q), w,v\in \ell^2(G)\}.
		\end{equation}
		We have then the following equality of Hilbert spaces:
		\begin{align}
			\label{eq:H123}
			H_1\oplus \tilde H_2\oplus H_3=&\left\{ 
			\begin{bmatrix} a+pw+b \\ -a+qv+b\end{bmatrix} \mid w,v,b\in \ell^2(G), a\in  \mathrm{im}(1-p)\cap\mathrm{im}(1-q)\right\} \\
			=& \left\{ 
			\begin{bmatrix} a+pw+b \\ a+pw+b-(2a+pw-qv) \end{bmatrix} \mid w,v,b\in \ell^2(G), a\in  \mathrm{im}(1-p)\cap\mathrm{im}(1-q)
			\right\}\\
			=&\left\{\begin{bmatrix}x\\ x-y\end{bmatrix} \mid x,y\in \ell^2(G) \right\} \\
			\cong & \ell^2(G)\oplus \ell^2(G) \label{eq5}.
		\end{align}
		Let us briefly explain how we obtain these. As $b\in \ell^2(G)$ is arbitrary, it implies that the first coordinate $a+pw+b$ of $H_1+H_2+H_3\subseteq \ell^2(G)\oplus \ell^2(G)$ can span the entire  $\ell^2(G)$. Additionally, by (\ref{eq:H}) the difference $2a+pw-qv$ of the two coordinates of $H_1+H_2+H_3\subseteq \ell^2(G)\oplus \ell^2(G)$ can also cover all of $\ell^2(G)$. Thus the third equality holds.
		%Let $\mathrm{im}(p)\oplus\mathrm{im}(q)=\left\{\begin{pmatrix}px\\ qy\end{pmatrix} \mid x,y\in l^2(G) \right\}$ be the subspace of $l^2(G)\oplus l^2(G)$. Consider the linear map \[
		%\mathrm{im}(p)\oplus\mathrm{im}(q)\rightarrow \mathrm{im}(p)+\mathrm{im}(q) \quad \begin{pmatrix}px\\ qy\end{pmatrix}\mapsto px+qy.
		%\]
		%The map is clearly surjective. 
		%Note that $\mathrm{im}(p)\cap\mathrm{im}(q)=\{0\}$. This implies that the map is injective. 
		%Here, the structure of the left regular representation on  $\ell^2(G)$ is needed. 	
		We replace $H_3$ in $H_1+ \tilde H_2+H_3$ by 
		\[
		\tilde H_3:=\left\{\begin{bmatrix}(1-p)b \\ (1-q)b \end{bmatrix} \mid  b\in \ell^2(G) \right\}
		\]
		to obtain an orthogonal direct sum. It is straightforward to check that $H_1, \tilde H_2, \tilde H_3$ are mutually orthogonal subspaces of $\ell^2(G)\oplus \ell^2(G)$ and that their direct sum $H_1\oplus \tilde H_2\oplus\tilde H_3$ is a subspace of $\ell^2(G)\oplus \ell^2(G)$.
		By (\ref{eq:H123})-(\ref{eq4}) every element of $\ell^2(G)\oplus \ell^2(G)$ can be represented by 
		\[
		\begin{bmatrix}a\\ -a\end{bmatrix}+\begin{bmatrix}pw\\ qv\end{bmatrix}+\begin{bmatrix}b\\ b\end{bmatrix},
		\]
		for some $w,v,b\in \ell^2(G)$ and some $a\in\mathrm{im}(1-p)\cap\mathrm{im}(1-q)$. The latter is equal to 
		\[
		\begin{bmatrix}a\\ -a\end{bmatrix}+\begin{bmatrix}p(w+b)\\ q(v+b)\end{bmatrix}+\begin{bmatrix}(1-p)b\\ (1-q)b\end{bmatrix}\in H_1\oplus H_2\oplus \tilde H_3.
		\]
		Hence $\ell^2(G)\oplus \ell^2(G)\subseteq H_1\oplus \tilde H_2\oplus \tilde H_3$.
		This finishes the proof.
	\end{proof}

Let us remark that $P_{H_3}$ and $P_{\tilde H_3}$ belong to $\mathrm M_2(\Csr(G))$. 
First observe that \[P_{ H_3}=\begin{bmatrix}\frac12 & \frac12 \\ \frac12 & \frac12 \end{bmatrix}\in \mathrm M_2(\Csr(G)).\]
Recall that by our assumption $\Delta_1$ has a spectral gap $(0, c]$ for some $c\in \mathbb R_{+}$.
Choose a continuous function on $\mathbb R$ such that $f$ vanishes when $x\le0$ and equals $1$ when $x>\frac{c}{2}$.
Then $P_{\ker\Delta_1}=f(\Delta_1)\in \mathrm M_2(\Csr(G))$. 
The invertible matrix $V\in \mathrm M_2(\Csr(G))$ 
from the proof of Lemma~\ref{lem:projker} conjugates  $P_{H_1}$ to $P_{\ker\Delta_1}$. 
Hence $P_{H_1}\in \mathrm M_2(\Csr(G))$. Consider $P_{\tilde H_2}=\begin{bmatrix}p & 0 \\0 & q\end{bmatrix}\in \mathrm M_2(\Csr(G))$. 
By Lemma \ref{lem:directsum} then
we have 
\[P_{\tilde H_3}=1-P_{\tilde H_2}-P_{H_1} \in \mathrm M_2(\Csr(G)).
\]
%Moreover, \[P_{ H_3}=\begin{bmatrix}\frac12 & \frac12 \\ \frac12 & \frac12 \end{bmatrix}\in \mathrm M_2(\Csr(G)).\]
The next lemma identifies the K-classes of $P_{H_3}$ and $P_{\tilde H_3}$ in $K$-theory.	
	\begin{lemma}
		\label{lem:H3}
		Let the Hilbert spaces $H_3$ and $\tilde{H}_3$ be as in (\ref{eq7}). The projections associated with them satisfy $$[P_{H_3}]=[P_{\tilde H_3}]\in \mathrm K_0(\Csr(G)).$$
	\end{lemma}
%	\todo{$P_{H_3}$, $P_{\tilde H_3}$}
	\begin{proof}
		First note that by Lemma \ref{lem:directsum}, the projection $P_{H_2}=1-(P_{H_1}+P_{H_3})$ belongs to $\mathrm M_2(\Csr(G))$. So by previous arguments we have that $P_{H_i}, P_{\tilde H_j}\in \mathrm M_2(\Csr(G))$ for $i=1,2,3$ and $j=1,2$.  
		To show that the projections associated with $H_3$ and $\tilde{H}_3$ have the same K-class, we construct an invertible operator in $\mathrm M_2(\Csr(G))$. Consider the following bijective continuous linear maps associated with $H_2$ and $\tilde{H}_2$ as well as $H_3$ and $\tilde{H_3}$.
		\begin{align*}
		u: H_2&\rightarrow \tilde H_2 \quad \begin{bmatrix}pw\\ qv\end{bmatrix}\mapsto \begin{bmatrix}pw+qv\\-pw-qv\end{bmatrix}
		\\
		v: H_3&\rightarrow \tilde H_3 \quad \begin{bmatrix}x \\ x\end{bmatrix}\mapsto \begin{bmatrix}(1-p)x\\(1-q)x\end{bmatrix}
		\end{align*}
	    Note that injectivity is due to the fact that $\mathrm{im}(p)\cap\mathrm{im}(q)=\{0\}$.
	    Together with the fact that $H_i, \tilde H_i$ are Hilbert spaces, the algebraic inverse of $u,v$ are also bounded. 
	    Employing these operators we construct an invertible operator $U$ in 
	    $\mathcal{B}(\ell^2(G)^{\oplus 2})$.
	    \[
	    	U=\begin{bmatrix}1 & 0 & 0\\ 0 & u & 0 \\ 0 & 0 & v\end{bmatrix}: H_1\oplus H_2\oplus H_3\rightarrow H_1\oplus\tilde H_2\oplus\tilde H_3.
	    \]
The operator $U$ belongs to $M_2(\Csr(G))$ since we can view it as
\[
U = P_{H_1} + P_{\tilde{H}_2} \begin{bmatrix} 1 & 1 \\ -1 & -1 \end{bmatrix} P_{H_2} + P_{\tilde{H}_3} \begin{bmatrix} 1-p & 0 \\ 0 & 1-q \end{bmatrix} P_{H_3}.
\]
Since $M_2(\Csr(G))$ is a $C^*$-algebra, it is closed under holomorphic functional calculus, which implies it is inverse closed. Hence, the inverse $U^{-1}$ of $U$ belongs to $M_2(\Csr(G))$ as well. 

To conclude the equality of K-theory classes, consider the element $e := U^{-1} P_{\tilde{H}_3} U \in M_2(\Csr(G))$. Since $U$ is not unitary, $e$ is generally not an orthogonal projection, but it is an idempotent. Because $U$ maps $H_3$ bijectively onto $\tilde{H}_3$, the image of the idempotent $e$ is precisely $U^{-1}(\tilde{H}_3) = H_3$. In a $\Cs$-algebra, any idempotent defines the same $\mathrm K_0$-class as the orthogonal projection onto its image, which gives $[e] = [P_{H_3}]$. Furthermore, since $e$ and $P_{\tilde{H}_3}$ are conjugate idempotents, they define the same K-theory class. Therefore, we obtain:
\[
[P_{\tilde{H}_3}] = [e] = [P_{H_3}] \in \mathrm K_0(\Csr(G)).
\]        
	\end{proof}
	
	Now we are ready to present our proof for Theorem \ref{thm: Z_n*Z_m}.
	
	\begin{proof}
		[Proof of Theorem~\ref{thm: Z_n*Z_m}]
		 We know by Lemma \ref{lem:existence} that $p_1$ exists. 
		%and its K-class satisfies  
		%\[
		 %	\tau([p_1]) = \beta^{(2)}_1(G) = 1 - \frac{1}{m} - \frac{1}{n}.
		%\]
		%We choose the Laplacian $\Delta_1$ described in Lemma \ref{lem:Delta1}.
		%We identify $P_{ker\Delta_1}$ by $p_1$. Denote $l^2(G)$ by $H$. 
		Lemma~\ref{lem:projker} implies that 
		$[p_1]\in \mathrm K_0(\Csr(G))$ is represented by the projection from $\ell^2(G) ^{\oplus 2}$ onto $H_1$. 
		In other words, we have \begin{equation*}
			%\label{eq:mid}
			[p_1]=[P_{H_1}]\in \mathrm K_0(\Csr(G)).
		\end{equation*}		
		 We infer from Lemma~\ref{lem:directsum} that
		\[
			[H_1]=[H\oplus H]-[\tilde H_3]-[pH]-[qH],
		\]
		which is by Remark \ref{rem:identification} equal to 
		\begin{equation*}
			%\label{eq:H1}
			[P_{H_1}]=2[1]-[P_{\tilde H_3}]-[p]-[q]\in \mathrm K_0(\Csr(G)).
		\end{equation*}
		Note that $[P_{\tilde H_3}]=[1]$. 
		Putting all these together we conclude that
		\[
			[p_1]=2[1]-[1]-[p]-[q]=[1]-[p]-[q]
			\in 
			\mathrm K_0(\Csr(G)).
		\]
		More concretely, we have 
		\[
		[p_1] = [1] - \left[\frac{1}{m}
		\sum_{0 \leq i < m}{s^i}\right] - 
		\left[\frac{1}{n}
		\sum_{0 \leq j <n}{t^j}\right].
		\]		
	\end{proof}

	When $m=2$ and $n=3$ the above proof can be slightly modified so that $\ker\Delta_1$ can be expressed more concretely. We discuss this in the following example.
	
	\begin{example}
		\label{example:PSL(2,Z)}
		Consider the group $\mathrm{PSL}(2, \mathbb Z) = \mathbb Z_2 \ast \mathbb Z_3$. The K-class of the higher Kazhdan projection $p_1$ can be described as
		\[
		[p_1] = [1]-\left[\frac{1+s}{2}\right]-\left[\frac{1+t+t^2}{3}\right]\in \mathrm K_0(
		\Csr(\mathrm{PSL}(2, \mathbb Z))).
		\]
	\end{example}
	
	\begin{proof}
		Let $s, t$ be generators of order 2 and 3, respectively. 
		%This obtained from a cochain complex associated with the free resolution of $\mathbb Z$ as a trivial $\mathbb ZG$-module.
		%Note that $G$ has spectral gap around 0 in the left regular representation. We verify this by applying Lemma \ref{lem:spectral gap condition}. Firstly, $\mathrm H^1(G, \ell^2 G)$ is reduced because $G$ is not amenable and secondly, $H^2(G, \ell^2 G)$ is reduced by the Shapiro lemma, since $\mathrm {PSL(2, \mathbb Z)}$ has the free group $\mathbb F_2$ as an index 6 subgroup. That guarantees that $p_1 \in \mathrm M_k(\Csr(\mathrm {PSL(2, \mathbb Z)}))$. In order to find the Laplacian $\Delta_1$ we need a cochain complex. Let us take a free resolution associated to the classifying space $\mathrm B G$ for the trivial $\mathbb ZG$-module $\mathbb Z$. 
		%We shall directly find the expression for elements in the kernel $\ker \Delta_1$. 
		By Lemma \ref{lem:Delta1}, the Laplacian $\Delta_1$ can be expressed as
		\[ 
		\Delta_1=d_0d_0^*+d_1^*d_1=
		\begin{bmatrix}4 & (1-s)(1-t^2)\\
			(1-t)(1-s) & 5+2t+2t^2
		\end{bmatrix}.
		\]
		As before set $p=\frac{1+s}{2}$ and $q=\frac{1+t+t^2}{3}$. 
		Similar to (\ref {eq3}) and (\ref{eq4}), we write
		\[
		1-p=\frac{1-s}{2}\frac{1-s}{2}, 
		\quad
		1-q=\frac{1-t}{\sqrt 3}\frac{1-t^2}{\sqrt 3}.
		\]
		View $\Delta_1$ as
		\[
		\Delta_1
		=
		\begin{bmatrix}4p & 0\\0 & 9q \end{bmatrix}
		+
		\begin{bmatrix} 
			\frac{1-s}{2} & 0 \\ 0 & \frac{1-t}{\sqrt 3}
		\end{bmatrix}
		\begin{bmatrix}
			4 & 2\sqrt 3 \\ 2\sqrt 3 & 3
		\end{bmatrix}
		\begin{bmatrix}\frac{1-s}{2} & 0 \\ 0 & \frac{1-t^2}{\sqrt 3}
		\end{bmatrix}.
		\]
		Following the argument as in Lemma \ref{lem:SolDelta1} and Lemma \ref{lem:projker} we have 
		\begin{align*}
			\ker\Delta_1=&\left\{
			\begin{bmatrix}1-p & 0\\ 0& 1-q\end{bmatrix}z \mid \begin{bmatrix}\frac{1-s}{2} & 0\\ 0& \frac{1-t}{\sqrt 3}\end{bmatrix}z=\begin{bmatrix}-\frac{\sqrt 3}{\sqrt 7}a\\ \frac{2}{\sqrt 7}a\end{bmatrix}
			a\in \mathrm{im}(1-p)\cap\mathrm{im}(1-q)
			\right\}\\
			=&
			\left\{
			\begin{bmatrix}\frac{1-s}{2}(-\frac{\sqrt 3}{\sqrt 7}a) \\ \frac{1-t^2}{\sqrt 3}\frac{2}{\sqrt 7}a \end{bmatrix} \mid 
			a\in\mathrm{im}(1-p)\cap\mathrm{im}(1-q) 
			\right\}.
		\end{align*}
	We infer that
		$[P_{\ker\Delta_1}]$ has the same  $K$-class as  the projection onto 
		\[
		H_1 =  \left\{
		\begin{bmatrix}(1-p)a\\ -(1-q)a\end{bmatrix} \mid a\in \mathrm{im}(1-p)\cap\mathrm{im}(1-q)
		\right\}.
		\]
		Therefore, we have
		\[
		[p_1]=[1]-[p]-[q]\in \mathrm K_0(\Csr(\mathrm{PSL}(2, \mathbb Z))).
		\]
	\end{proof}

	As an immediate consequence of Theorem \ref{thm: Z_n*Z_m} we have the following result. 
	
	\begin{corollary} \label{cor: deloc l2 Betti Zn*Zm}
		The delocalised $\ell^2$-Betti numbers for $G = \mathbb Z_m \ast \mathbb Z_n $ are
		\begin{equation*}
			\beta^{(2)}_{1, \langle{g}\rangle}(G) = 
			\begin{cases}
				1-\frac{1}{m}-\frac{1}{n} &\qquad g=e\\
				-\frac{|\langle{g}\rangle|}{m}  &\qquad g \in \mathbb Z_m \setminus \{e\}\\
				-\frac{|\langle{g}\rangle|}{n}  &\qquad g \in \mathbb Z_n \setminus \{e\}\\
				0 \qquad & \qquad  \text{otherwise} 
			\end{cases}
		\end{equation*}
		and $\beta ^{(2)}_{k, \langle{g}\rangle}(G) = 0$ for $k \neq 1$ and $g\in G$.
	\end{corollary}
	
	\begin{proof}
		Due to hyperbolicity of $G$, the delocalised traces are well-defined on $\mathrm {K}_0(\Csr(G))$. The pairing between these traces and $[p_1]$ from Theorem \ref{thm: Z_n*Z_m} implies the desired result for delocalised $\ell^2$-Betti numbers in degree 1. In all other degrees this invariant vanishes. This is because by Lemma \ref{lem:existence}, we have $p_k = 0$ for $k\neq 1$.
	\end{proof}	

	In particular, when $m=2$ and $n=3$ we obtain the following result.	
	
	\begin{example} \label{example: deloc l2 Betti PSL2Z}
		The delocalised $\ell^2$-Betti numbers for $G = \mathrm{PSL}(2, \mathbb Z)$ are
		\begin{equation*}
			\beta^{(2)}_{1, \langle{g}\rangle}(G) = 
			\begin{cases}
				1/6 &\qquad g=e\\
				-1/2  &\qquad g\in \langle{s}\rangle\\
				-1/3  &\qquad g\in \langle{t}\rangle\\
				-1/3  &\qquad g\in \langle{t^2}\rangle\\
				0    &\qquad \text{otherwise} 
			\end{cases}
		\end{equation*}
		
		and $\beta ^{(2)}_{k, \langle{g}\rangle}(G) = 0$ for $k \neq 1$ and $g\in G$.
	\end{example}

	\subsection{Certain product groups}
	Consider a product of free groups with a finite group. We show that higher Kazhdan projections for such a group are not a multiple of the identity in any fixed degree. We further describe their K-classes. 
	% As an application we provide further non-vanishing results for delocalised $\ell^2$-Betti numbers. In particular, we establish non-vanishing results in all degrees for this invariant.
	
	We recall from \cite[Example 1.5.1]{linowakpooya2020} an argument showing that the K-class of $p_1$ in $\mathrm K_0(\Csr(\mathbb F_2))$ for the free group $\mathbb F_2$ is $[1]$. To view $p_1$ as a member of a matrix algebra over $\Csr(\mathbb F_2)$, a spectral gap for $\Delta_1$ is required.	For this we appeal to Lemma \ref{lem:spectral gap condition}. In fact,
	the cohomology group $\mathrm H^1(\mathbb F_2, \ell^ 2 (\mathbb F_2))$ is reduced because $\mathbb F_2$ is non-amenable, and
	$\mathrm H^2(\mathbb F_2, \ell^ 2 (\mathbb F_2))$ is reduced (even vanishes) because the classifying space $\mathrm{B}\mathbb F_2$ has a one-dimensional model (e.g.\ its presentation complex). Hence $\Delta_1$ has a spectral gap.
	%Hence $p_1$ lies in some matrices over $\Csr (\mathbb F_2)$.
	Now since 
	$$\mathrm K_0(\Csr (\mathbb F_2)) =  \langle{[1]}\rangle \cong \mathbb Z,$$ we conclude that $[p_1] = k [1]$, for some $k\in \mathbb Z$. 
	Equation 
	(\ref{def:l2 betti number definition}) together with $\beta^{(2)}_ 1(\mathbb F_2)=1$ imply that 
	$$k=\tau([p_1]) =1,$$ 
	hence $[p_1] = [1]$. It follows that $p_1$ is even Murray-von Neumann equivalent to $1$.  See \cite{dykemahaageruprordam1997}
%stable rank one argument

	\begin{lemma}\label{lem:spectral gap for productof groups}
		Let $G$ be a finite product of groups $G = G_1 \times \cdots \times G_k$. Assume that for $1\leq j \leq k$ there is a Hilbert chain complex such that the associated Laplacians $\Delta_i^{G_j}$ has spectral gap for all $i\in \mathbb N$. Then the Laplacians $\Delta_i ^{G}$ associated to the product complex of $G$ have spectral gap for all $i\in \mathbb N$.

		%Let $G = G_1\times \dotsc \times G_n$ be a product of virtually free groups. Assume that $\Delta_j^{G_i}$ has spectral gap for $j=0,1$ and all $1\leq i \leq n$. Then $\Delta^{G}_n$ has spectral gap in all degree below $n$. 
	\end{lemma}

    \begin{proof}
    	The result follows by induction once we prove it for the product of two groups $G = G_1 \times G_2$. 
    	%Let $n\in \mathbb N$. 
    	Let $(C^{j,i})_{i\in \mathbb N}$ be Hilbert cochain complexes whose associated Laplacians $(\Delta_i^{G_j})_i$ have spectral gap. Then the coboundary maps of the associated product complex
    	$\bigoplus_{r+s=i} C^{1,r} \hat{\otimes} C^{2,s}$
    	are given by
    	\[
    		d^i = \bigoplus_{r+s=i} d^r \otimes I + (-1)^s I \otimes d^s,
    	\]
    	and in turn the Laplacian $\Delta^G_i$ is given by 
    	\begin{equation*} \label{equ:Delta_n}
    		\Delta^G_i = \bigoplus_{r+s=i} \Delta^{G_1}_r \otimes I + I \otimes \Delta^{G_2}_s.
    	\end{equation*}
     	By construction a Laplacian is a positive operator. Moreover, by our assumption, $\Delta_j ^{G_i}$'s have spectral gap for all $i\in \mathbb N$. Therefore, for all $i\in \mathbb N$ the Laplacian $\Delta_i ^ {G}$ is a sum of commuting positive operators with spectral gap and hence has spectral gap. 	
    	%By our assumption that $\Delta_n^{G_i}$'s have spectral gap for all $n \in \mathbb N$, we conclude that so does $\Delta_n^{G}$.
    	    	%We only show it for $G= G_1 \times G_2$ as the very same argument apply to the product with $n$-terms. We will work with cellular cohomology. Since cells in the products are products of cells of lower dimensions, there is a unitary isomorphism between the cochain spaces 
    			%$$ C^m(G_1 \times G_2, \ell^2(G_1 \times G_2)) \quad \text{and} \quad 
    			%\bigoplus_{r+s=m} C^r(G_1, \ell^2(G_1)) \otimes C^s(G_2, \ell^2(G_2)).$$ The coboundary $d_m$ is given by the formula
    			%\[
    			%	d_m = \bigoplus_{r+s=m} d_r \otimes I + (-1)^s I \otimes d_s,
    			%\]
    			%	and in turn the Laplacian $\Delta_m$ is given by the formula 
    			%\begin{equation} \label{equ:Delta_n}
    			%	\Delta_m = \bigoplus_{r+s=m} \Delta_r \otimes I + I \otimes \Delta_s.
    			%\end{equation}
				%		The Laplacian $\Delta^G_0$ has spectral gap, because all $\Delta_0^{G_i}$ have spectral gap by non-amenability of $G_i'$s.
   			 	%	For any degree $ m\leq 2$, $\Delta^G_m$ can be written as a combination of zeroth and first degrees Laplacians. For $m> 2$, $\Delta^G_m =0$ as the coboundarys are zero. By definition Laplacian is a positive operator.
    		%Since $\Delta_m$ is the direct sum of commuting positive operators with spectral gap, it has spectral gap.     
    		% \todo{Lemma: spectral gap for Laplacian is independent of a model / / or resolution we start with}		
    \end{proof}	
	
	\begin{theorem} \label{thm:higher Kazh proj product group}
		Let $F$ be a finite group, let $n \in \mathbb N$ and consider the product
		$G = \mathbb F_2 \times \dotsc \times \mathbb F_2 \times F$ of $n$ factors of $\mathbb F_2$ with $F$. Then the K-class of the $n$-th higher Kazhdan projection $p_n$ of $G$ satisfies		
		\[
			[p_n] = \left[\frac{1}{|F|}{\sum_{g\in F}{g}}\right].
		\]
	\end{theorem}
	
	\begin{proof}
		Recall that the product resolution is calculating the cohomology of a product of groups. We prove the following statement by induction on the number $n$ of factors of $\mathbb F_2$. It immediately implies the theorem. For all $n\in \mathbb N$, the Laplacian associated with the product resolution of $n$ factors of $\mathbb{F}_2$ with $F$ has spectral gap in all degrees and the $n$-th Kazhdan projection is Murry-von Neumann equivalent in $\mathrm M_{\infty}(\Csr(G)) = \mathrm {M_{\infty}}(\Csr(\mathbb F_2) \otimes  \cdots \otimes \Csr(\mathbb F_2) \otimes \Csr(F))$ to 
		\[
			1\otimes \cdots \otimes 1\otimes \frac{1}{|F|}{\sum_{g\in F}{g}}.
		\]
		
		The claim about a spectral gap is the content of Lemma \ref{lem:spectral gap for productof groups}. We prove the claimed Murray-von Neumann equivalence.  The base case $n=0$ is trivial by a known calculation. Assume that the statement holds for $n$ factors of $\mathbb F_2$, we show it for $n+1$ factors. 
		Consider the product decomposition 
		\[
			G = \mathbb F_2 \times (\underbrace{\mathbb F_2 \times \dotsc \times \mathbb F_2}_{n \text{ times}} \times F) := H \times K.
		\]
		Considering the iterated product resolution, we have
		\[
			 C^{n+1}(H \times K, \ell^2(H \times K)) \cong 
			\bigoplus_{r+s=n+1} C^r(H, \ell^2(H)) \otimes C^s(K, \ell^2(K)).	
		\]
		%we are only interested in degrees $m= n, n+1, n+2$, as $\Delta_{n+1}^G$ is defined in terms of the coboundaries of these cochain spaces.
		For $r\geq 2$, $C^r(H, \ell^2(H))$ is zero, so the corresponding term in the direct sum vanishes. Moreover, when $s \geq n + 1$, the same happens to $C^s(K, \ell^2(K))$. Therefore, we find that
		\[
		C^{n+1}(H \times K, \ell^2(H \times K)) \cong C^1(H, \ell^2(H)) \otimes C^n(K, \ell^2(K))	
		\]
		and the associated Laplacian is 
		\[
		\Delta^G_{n+1} = \Delta_1^H \otimes I + I \otimes \Delta_n^K
		.
		\]
		The kernel of this sum of positive commuting operators is equal to the intersection of the kernels of $\Delta_1^H \otimes I$ and $I \otimes \Delta_n^K$. So the kernel projection is $p_{n+1}^G = (p_1^{H} \otimes 1) \cdot (1 \otimes p_n^K) = p_1^{H} \otimes p_n^K$.
		
		The induction hypothesis says that there is a partial isometry $v$ whose support projection is $p_n^{K}$ and whose range projection is 
		$1 \otimes \cdots \otimes 1\otimes \frac{1}{|F|}{\sum_{g\in F}{g}}$ with $n$ factors of the unit $1 \in \Csr(\mathbb F_2)$. Since $p_1^{H}$ is Murray-von Neumann equivalent to $1$, there is a partial isometry $u \in \mathrm M_\infty(\Csr(\mathbb F_2))$ whose support projection is $p_1^H$ and whose range projection is $1$. Considering the partial isometry $u\otimes v$, we have that
		\[
			p_{n+1}^{G}  \stackrel{MvN}{\sim} \underbrace{1\otimes \cdots \otimes 1}_{n+1 \, {\text{times}}}\otimes \frac{1}{|F|}{\sum_{f\in F}{f}}.
		\]
		This finishes the proof.
  \end{proof}

	\bibliographystyle{alpha}
	\bibliography{mybib}

	%%%%%%%%%%
	%% Address
	%%%%%%%%%%
	\vspace{2em}
	\begin{minipage}[t]{0.45\linewidth}
		\small
		Sanaz Pooya \\
		Institute of Mathematics\\
		University of Potsdam\\
		14476 Potsdam, Germany\\
		{\footnotesize sanaz.pooya@uni-potsdam.de}
		\\   
		%	 Second author's address
		\\
		\small 
		Hang Wang \\ Research Center of Operator Algebras \\
		East China Normal University\\
		Shanghai 200241, China \\
		{\footnotesize wanghang@math.ecnu.edu.cn}
	\end{minipage}
\end{document}